\documentclass[12pt,reqno]{amsart}

\usepackage{amsmath, amssymb, amsfonts, mathrsfs, bm, enumerate, mathtools}
\topmargin by -\headsep

\newtheorem{theorem}{Theorem}[section]
\newtheorem{lemma}[theorem]{Lemma}
\newtheorem{corollary}[theorem]{Corollary}

\theoremstyle{definition}

\theoremstyle{remark}

\numberwithin{equation}{section}

\def\a{{\mathbf a}}
\def\b{{\mathbf b}}

\def\e{{\mathbf e}}

\def\j{{\mathbf j}}
\def\k{{\mathbf k}}

\def\p{{\mathbf p}}

\def\u{{\mathbf u}}

\def\x{{\mathbf x}}

\def\A{{\boldsymbol \alpha}}
\def\B{{\boldsymbol \beta}}
\def\G{{\boldsymbol \gamma}}
\def\D{{\boldsymbol \delta}}
\def\L{{\boldsymbol \lambda}}
\def\LL{{\boldsymbol \Lambda}}
\def\ZZ{{\boldsymbol \zeta}}

\def\M{{\boldsymbol \mu}}

\def\X{{\boldsymbol \xi}}

\def\cA{{\mathcal A}}
\def\cB{{\mathscr B}}

\def\cI{{\mathcal I}}
\def\cJ{{\mathcal J}}

\def\cM{{\mathcal M}}
\def\cN{{\mathcal N}}

\def\cR{{\mathcal R}}

\def\sA{{\mathscr A}}
\def\sC{{\mathscr C}}

\def\setmin{{\setminus\!\,}}

\def\tk{\tilde{k}}

\def\tN{\widetilde N}

\def\R{{\mathbb R}}\def\N{{\mathbb N}}
\def\Z{{\mathbb Z}}\def\Q{{\mathbb Q}}

\def\fm{{\mathfrak m}}
\def\fM{{\mathfrak M}}
\def\fN{{\mathfrak N}}
\def\fn{{\mathfrak n}}
\def\fS{{\mathfrak S}}

\def\fB{{\mathfrak B}}
\def\fC{{\mathfrak C}}

\def\fJ{{\mathfrak J}}

\def\l{\ell}
\def\d{\mathbf d}
\def\ep{\varepsilon}
\def\implies{\Rightarrow}

\def\tG{\widetilde G}
\def\le{\leqslant}
\def\ge{\geqslant}
\def\ts{\tilde{s}}

\def\dr{{\rm d}}

\DeclareMathOperator{\vol}{vol}

\begin{document}

\title[Simultaneous additive equations]{Simultaneous additive equations: \\Repeated and differing degrees}

\author[J. Brandes]{Julia Brandes}
\address{JB: Mathematisches Institut,
Bunsenstr. 3--5, 37073 G\"ottingen, Germany, \newline
{\tt jbrande@uni-math.gwdg.de}}

\author[S. T. Parsell]{Scott T. Parsell}
\address{STP: Department of Mathematics,
West Chester University, 25 University Ave., West Chester, PA 19383, U.S.A.,
{\tt sparsell@wcupa.edu}}
\thanks{The second author is supported in part by National Security Agency Grant H98230-13-1-207}

\begin{abstract}
We obtain bounds for the number of variables required to establish Hasse principles, both for existence of solutions and for asymptotic formul\ae, for systems of additive equations containing forms of differing degree but also multiple forms of like degree.  Apart from the very general estimates of Schmidt and Browning--Heath-Brown, which give weak results when specialized to the diagonal situation, this is the first result on such ``hybrid" systems.
\end{abstract}

\maketitle

\section{Introduction}

When $c_{ij}$ are nonzero integers and $d_j$ are natural numbers with $d_1 \ge \dots \ge d_r$, we consider the solubility of the general system of additive forms
\begin{align}\label{sys}
    \sum_{j=1}^s c_{ij} x_j^{d_i} = 0 \quad (1 \le i \le r)
\end{align}
in integers $x_1, \dots, x_s$.  There is a fundamental dichotomy in the strategy for handling such systems, which depends on whether all forms are of the same degree.  When the degrees are the same, the classical approach is to make a linear change of variables so that the mean values factor into a product of one-dimensional integrals, as in the work of Davenport and Lewis \cite{DL:69}, Cook \cite{Cook:72}, \cite{Cook:83}, and Br\"udern and Cook \cite{BC:92}, though recently new ideas have become available in the work of Br\"udern and Wooley \cite{BW:07,BW:13,BW:14h3,BW:14corr}.
Meanwhile, when the $d_j$ are distinct, such investigations are made possible by the iterative method of Wooley \cite{W:91sae2}, \cite{W:97essn}, \cite{W:98sae4}, which yields mean value estimates for exponential sums of the shape
\begin{align*}
    f_{\k}(\A; \cA) = \sum_{x \in \cA} e(\alpha_1 x^{k_1} + \dots + \alpha_t x^{k_t})
\end{align*}
when $\cA$ is a set of suitably smooth integers. Here the second author \cite{P:02pae} has obtained bounds for pairs of equations in a handful of particular cases by optimizing over a large collection of iterative schemes in the style of Vaughan and Wooley \cite{VW:95}. In \cite{P:02ineq3}, these results were extended to pairs of diophantine inequalities and to more general mixed systems with all degrees distinct.

Additive systems in which some, but not all, of the degrees are repeated would seem to require a hybrid of the two approaches, and the purpose of this paper is to present such a strategy.  The bounds we ultimately obtain are in line with what might be expected, given the results discussed above.

It is convenient for the analysis to sort the equations in \eqref{sys} by placing the various degrees in order of decreasing multiplicity.
For $1 \le l \le t$ write $k_{l,1},\dots,k_{l,\nu(l)}$ for those distinct values of the exponents $d_1,\dots,d_r$ that occur with the same multiplicity $\mu_l$. Plainly, we may suppose that $\mu_1 > \dots > \mu_t$. Write further $\rho_l = \mu_l \nu(l)$ and
\begin{align}\label{r}
    r_{l,n}= \mu_1 \nu(1) + \dots + \mu_{l-1} \nu(l-1) + \mu_l n \quad (1 \le n \le \nu(l), 1 \le l \le t)
\end{align}
and let $r_l = r_{l, \nu(l)}$, so that $r_{t}= \rho_1+\dots+\rho_t =r$, and with the conventions that $r_{l,0}= r_{l-1}$ and $r_{0,0}=0$. We adopt the notation
\begin{align*}
    \cI_{l,n} &= [r_{l,n-1}+1,r_{l,n}] \quad (1 \le n \le \nu(l), 1 \le l \le t), \\
    \cI_{l} &= [r_{l-1}+1,r_{l}] \quad (1 \le l \le t).
\end{align*}

After re-arranging the equations, we may then further suppose that the system takes the shape
\begin{align*}
    \sum_{j=1}^s c_{ij} x_j^{k_{l,n}} = 0 \quad (i \in \cI_{l,n}, 1 \le n \le \nu(l), 1 \le l \le t).
\end{align*}
We write $K_l = k_{l,1}+\dots + k_{l,\nu(l)}$, and
\begin{align*}
    K=d_1 + \dots + d_r = \mu_1K_1 + \dots + \mu_t K_t
\end{align*}
for the total degree of the system \eqref{sys}. We further write $M=\mu_1$ and adopt the convention that $\mu_0=\mu_{t+1}=0.$
We note that the two viewpoints $(d_1, \dots, d_r)$ and $(\k; \M)$ of organizing the degrees of the forms appearing in the system are both occasionally useful, so we retain both notations.

In most cases, the number of variables required to establish local solubility in the current state of technology (see for example the work of Knapp \cite{Knapp:07}) is larger than what is needed to establish a local-global principle via the circle method, so we focus our attention on the latter problem.  We aim for two types of Hasse principles, one for existence of solutions and one for asymptotic formul\ae. For the problem concerning existence of solutions, we make use of smooth number technology.  Write
\begin{align*}
    \cA(P,R) = \{n \in [1,P]: p|n, p \ {\rm prime}  \implies p \le R\}
\end{align*}
for the set of $R$-smooth numbers up to $P$.  Throughout, we fix $R=P^{\eta}$ for some sufficiently small positive number $\eta = \eta(s,\d)$.

We say that the system \eqref{sys} is {\it highly non-singular} if for every $1 \le n \le \nu(l)$ and every $1 \le l \le t$ one has
\begin{align}\label{rankcond}
    \det(c_{ij})_{i \in \cI_{l,n},j \in \cJ_l} \neq 0
\end{align}
for every $\mu_l$-tuple $\cJ_l \subseteq \{1,\dots,s\}$.  At the cost of a few extra variables, one may replace this condition by a weaker but more complicated rank condition across the blocks of variables defined in Section 2 below. Since both conditions are satisfied by almost all systems of the shape (\ref{sys}), we choose the former hypothesis for its simplicity and for the additional flexibility it affords us in the analysis.

Define
\begin{align}\label{def-varpi}
  \varpi_h = \nu(1) + \dots + \nu(h).
\end{align}
For any vector of distinct natural numbers $\k_h = (k_{l,n})_{\substack{1 \le n \le \nu(l)\\ 1 \le l \le h}}$,
write
\begin{align*}
    \tk_h=\max_{\substack{1 \le n \le \nu(l)\\1 \le l \le h}}\{k_{l,n}\}, \quad k=\tk_t, \quad \mbox{and} \quad \mu=\mu_t=\min\{\mu_1, \dots, \mu_t\}.
\end{align*}
Furthermore, let $u_0(\k_h)$ denote the least integer $u$ with the property that
\begin{align*}
    \int_{[0,1)^{\varpi_h}} |f_{\k_h}(\G; \cA(P,R)) |^{2u} \, \dr \G \ll P^{2u-(K_1+ \dots + K_h)}.
\end{align*}
Let $G^*(d_1, \dots, d_r) = G^*(\k; \M)$ denote the smallest integer $s$ for which every highly non-singular system (\ref{sys}) has the property that there exists a nontrivial positive integer solution whenever there exist non-singular positive real solutions and non-singular $p$-adic solutions for all primes $p$.
Similarly, let $v_0(\k_h)$ denote the least integer $v$ with the property that
\begin{align*}
    \int_{[0,1)^{\varpi_h}} |f_{\k_h}(\G; [1,P]) |^{2v} \, \dr \G \ll P^{2v-(K_1+ \dots + K_h)+\ep}.
\end{align*}
Then write $\tG^*(d_1, \dots, d_r) = \tG^*(\k; \M)$ for the analogous number of variables required (under the same local solubility hypotheses) to show that the number of solutions $\x \in [1,P]^s$ of every highly non-singular system is given by
\begin{align}\label{N-asymp}
    \cN(P)=(C+o(1))P^{s-K}
\end{align}
for some positive constant $C=C(s,\d)$.

\begin{theorem}\label{mainthm}
    \begin{enumerate}[(A)]
	\item \label{exist}
	    Let $s(\k_h)=\max\{u_0(\k_h),\frac{1}{2}k(1+\varpi_h)\}$ for $1 \le h \le t$. Then one has
	    \begin{align*}
		G^*(\k; \M) \le 2\sum_{h=1}^{t}(\mu_h-\mu_{h+1}) s(\k_h) + M.
	    \end{align*}
	\item \label{asymp}
	    Let $\ts(\k_h)=\max\{v_0(\k_h),\frac{1}{2}k(1+\varpi_h)\}$ for $1 \le h \le t$. Then one has
	    \begin{align*}
		\tG^*(\k; \M) \le 2\sum_{h=1}^{t}(\mu_h-\mu_{h+1}) \ts(\k_h) + 1.
	    \end{align*}	
    \end{enumerate}
\end{theorem}

Apart from general results of Birch \cite{Bir:57} and Schmidt \cite{Sch:85} and recently Browning and Heath-Brown \cite{BHB:14}, which apply to more general (non-diagonal) systems, the bound in Theorem~\ref{mainthm} is the first of its kind, in which the diagonal structure is exploited to obtain competitive bounds on the number of variables required.  We note that, in the presence of sufficiently strong mean value estimates so that the above maxima were $\tfrac{1}{2}k(1+\varpi_h)$ for all $h$, the bounds in (A) and (B) would become $k(M+r) + M$ and $k(M+r)+1$, respectively.  While conclusions of such strength are currently beyond our grasp, Theorem~\ref{mainthm} can be made explicit by inserting bounds from the literature.
In particular, by applying the results of Wooley \cite[Theorem 1.2]{W:14mec3} one obtains the following.

\begin{corollary}\label{maincor}
    Suppose that $d_i \ge 3$ for all $i$, and write $\ts_1(\k_h)=\max\{\tk_h(\tk_h-1),\frac{1}{2}k(1+\varpi_h)\}$ for $1 \le h < t$.  Then one has
    \begin{align*}
	\tG^*(\k; \M) \le 2\mu  k(k-1) + 2\sum_{h=1}^{t-1}(\mu_{h}-\mu_{h+1}) \ts_1(\k_h)   + 1.
    \end{align*}
\end{corollary}
Here we are able to take $v_0(\k_h) = \tk_h(\tk_h-1)$ in Theorem \ref{mainthm} (B), and in this instance one easily verifies that $v_0(\k_t)=k(k-1)$ exceeds $\tfrac{1}{2}k(1+\varpi_t)$ whenever $k \ge 3$.
Similarly, the results of Wooley \cite{W:97essn} (see also \cite[Corollary 1.3]{P:02ineq3}) show that one has $u_0(\k_h) \le (1+o(1))H(\k_h)$, where
\begin{align*}
    H(\k_h) = \tk_h \varpi_h(\log \tk_h + 3\log \varpi_h),
\end{align*}
with refined conclusions available for various ranges of the parameters. One may therefore derive bounds analogous to Corollary \ref{maincor} for the function $G^*(\k;\M)$.
We highlight in particular some consequences of our results for the simplest collections of exponents not covered by previous work.

\begin{corollary}\label{kncor}
    Let $k$ and $n$ be integers with $k > n \ge 2$. Then one has the bounds
    \begin{align*}
	\begin{array}{c}
	    \tG^*(k,k,n) \le 4k(k-1)+1, \quad \tG^*(k,k,n,n) \le 4k(k-1)+1, \\
	    \;\\
	    \tG^*(k,n,n) \le  \begin{cases}
				  2k(k-1) + 2n(n-1) + 1 & \text{ if } k \le n(n-1) \\
				  2k^2  + 1  & \text{ if } k \ge  n(n-1),
	                      \end{cases}
	\end{array}
    \end{align*}
    and
        \begin{align*}
	\begin{array}{c}
	    G^*(k,k,n) \le (6+o(1))k \log k, \quad G^*(k,k,n,n) \le (8+o(1))k \log k, \\
	    \;\\
	    G^*(k,n,n) \le  (4+o(1))k \log k + 2 n \log n.
	\end{array}
    \end{align*}
\end{corollary}
Observe that here it suffices to have $k > n \ge 2$, as in the results for $\tG^*(k,k,n)$ and $\tG^*(k,k,n,n)$ we use Wooley's bounds $v_0(k,n) \le k(k-1)$ only, which hold for all $k \ge 3$ regardless of the value of $n$. For $\tG^*(k,n,n)$ one needs additionally the bound $v_0(n) \le n(n-1)$, which holds for $n \ge 3$ by Wooley's bound as above and for $n = 2$ by Hua's Lemma.

%
%

While these bounds follow as a direct consequence of our more general estimates, one would expect that a more detailed analysis of these special cases should yield better results. In particular, the strategies of Wooley \cite{W:12war} for making the transition from complete Vinogradov-type systems to incomplete systems associated with Waring's problem have the potential to be employed here to a greater extent. Thus, we may expect some small improvements in the bounds for $v_0(\k)$, which we have estimated trivially by $v_0(1,2, \dots, k)$. In fact, we may illustrate the potential of our methods by considering certain systems of small degree.
\begin{theorem}\label{quadcubthm}
    For systems of $r_Q$ quadratic and $r_C$ cubic equations one has the bounds
    \begin{align*}
	\tG^*(2,3; r_Q,r_C) \le \begin{cases}
				      4 r_Q + \lfloor(20/3)r_C\rfloor + 1 & \text{ if } \ r_Q \ge r_C, \\
				      8 r_C + \lfloor(8/3)r_Q\rfloor + 1 & \text{ for } \ r_C \ge r_Q.
	                         \end{cases}
    \end{align*}
    Furthermore, for $r_C>r_Q$ we also have
    \begin{align*}
	G^*(3,2;r_C,r_Q) &\le 7 r_C + \lceil(11/3)r_Q\rceil.
    \end{align*}
\end{theorem}
Note that for systems of $r_Q$ quadratic forms and one cubic form Theorem~\ref{quadcubthm} yields
\begin{align*}
    \tG^*(2,3;r_Q,1) \le 4r_Q+7 = 2 \cdot(2r_Q+3)+1,
\end{align*}
so the bound achieves the square root barrier in this case, thus joining the small group of examples for which we are able to establish bounds of this quality.
Unfortunately, for other situations we do not obtain equally strong results, largely due to the lack of sufficiently powerful mean values. 
%
%
%
%
We also note that it may be possible to remove the explicit assumption of non-singularity for the real and $p$-adic solutions by adapting work of the first author \cite{Brandes:15p}. We intend to pursue some of these refinements in future papers.

In Section 2, we establish our main mean value estimate, and we then prove Theorem~\ref{mainthm} in Sections 3 and 4 by applying the circle method.  Finally, in Section 5 we establish a few auxiliary results that will allow us to refine our arguments to obtain the bounds advertised in Theorem~\ref{quadcubthm} for systems of cubic and quadratic equations in Section 6.

The authors are grateful to Trevor Wooley for many helpful conversations and suggestions.  This work originated with a visit of the first author to West Chester University and was facilitated by a subsequent workshop at Oberwolfach; the authors thank both institutions for their hospitality and support.

\section{The mean value estimate}

The following notational conventions will be observed throughout the paper. Any expression involving the letter $\ep$ will be true for any (sufficiently small) $\ep > 0$. Consequently, no effort will be made to track the respective  `values'  of $\ep$. Also, any statement involving vectors is to be understood componentwise. In this spirit, we write $(q,\b)=(q,b_1, \dots, b_n)$ whenever $\b \in \Z^n$, and we interpret a vector inequality of the shape $C \le \b \le D$ to mean that $C \le b_i \le D$ for $i=1, \dots, n$.

For $\A \in [0,1)^r$ define
\begin{align}\label{gamalp}
    \gamma_{j,l,n}= \sum_{i=r_{l,n-1}+1}^{r_{l,n}} c_{ij} \alpha_i \quad (1 \le j \le s, 1 \le n \le \nu(l), 1 \le l \le t)
\end{align}
and write $\G_j = (\gamma_{j,l,n})_{1 \le n \le \nu(l),1 \le l \le t}$.  Furthermore, set $f_j(\A; \cA) = f(\G_j;\cA)$ where $\cA = [1,P]$ or $\cA = \cA(P,R)$, with the convention that the explicit mention of the set $\cA$ will be suppressed whenever there is no danger of confusion. We partition the indices $\{1,\dots,s\}$ into $M+1$ blocks
\begin{equation}
\label{partition}
    \{1,\dots,s\} = \cB_0 \cup \bigcup_{h=1}^t \bigcup_{m=1}^{\mu_{h}-\mu_{h+1}} \cB_{h,m},
\end{equation}
where each block $\cB_{h,m}$ is of size $2u_h$ with any excess variables placed into the block $\cB_0$, and define
\begin{align}\label{s0def}
    s_0 = \sum_{h=1}^t (\mu_h-\mu_{h+1})u_h.
\end{align}
Consider the mean value
\begin{equation}
\label{Idef}
    I_{\u,\k, \M}(\cA) = \int_{[0,1)^r} \prod_{h=1}^t \prod_{m=1}^{\mu_h-\mu_{h+1}} \prod_{j \in \cB_{h,m}} f_j(\A; \cA)  \, \dr \A.
\end{equation}
This mean value may be bounded in terms of simpler mean values.

\begin{theorem}\label{mvt}
    For $\cA = [1,P]$ or $\cA = \cA(P,R)$ one has
    \begin{align*}
	I_{\u,\k, \M}(\cA) \ll \prod_{h=1}^t \left(J_{u_h, \k_h}(\cA)\right)^{\mu_h-\mu_{h+1}},
    \end{align*}
    where $J_{u_h,\k_h}(\cA)$ denotes the mean value
    \begin{align*}
	J_{u_h,\k_h}(\cA) = \int_{[0,1)^{\varpi_h}} |f_{\k_h}(\G;\cA)|^{2u_h} \, \dr \G.
    \end{align*}
\end{theorem}

In particular, this implies that we will have a perfect mean value estimate for $I_{\u, \k, \M}(\cA)$ as soon as we have perfect estimates for the primitive mean values $J_{u_h,\k_h}(\cA)$ for $1 \le h \le t$.

\begin{corollary} \label{mvtcor}
    Suppose $u_h$ is large enough that one has
    \begin{align*}
	J_{u_h,\k_h}(\cA) \ll P^{2u_h-(K_1+\dots+K_h)+\ep} \qquad (1 \le h \le t).
    \end{align*}
    Then
    \begin{align*}
	I_{\u,\k, \M}(\cA) \ll P^{2s_0-K+\ep}.
    \end{align*}
\end{corollary}

This follows from the theorem on observing that
\begin{align*}
    \sum_{h=1}^t (K_1+\dots +K_h)(\mu_h-\mu_{h+1}) = K_1\mu_1 + \dots + K_t \mu_t = K.
\end{align*}

\begin{proof}[Proof of Theorem~$\ref{mvt}$]
    Set $\cA = [1,P]$ or $\cA=\cA(P,R)$, and write
    \begin{align*}
	  s_0(h) = \sum_{l=1}^h (\mu_l-\mu_{l+1})u_l \quad (1 \le h \le t) \quad \text{and } \quad s_0(0)=0,
    \end{align*}
    so that $s_0(t)= s_0$.

    First of all, by making a trivial estimate and applying the trivial inequality
    \begin{align}\label{trivineq}
	|z_1 \cdots z_n| \le |z_1|^n + \dots + |z_n|^n,
    \end{align}
    we find that
    \begin{align}\label{diag}
	I_{\u,\k,\M}(\cA) \ll \int_{[0,1)^r} \prod_{h=1}^t \prod_{m=1}^{\mu_h-\mu_{h+1}} |f(\G_{j(h,m)}; \cA)|^{2u_h} \, \dr \A
    \end{align}
    for some $j(h,m) \in \cB_{h,m}$.  Observe that the mean value on the right hand side of \eqref{diag} counts solutions to the system
    \begin{align}\label{eq1}
	\sum_{h=1}^t \sum_{m=1}^{\mu_h-\mu_{h+1}} c_{i,j(h,m)} \xi_{h,m}(k_{l,n}) = 0 \quad (i \in \cI_{l,n}, 1 \le n \le \nu(l),1 \le l \le t),
    \end{align}
    where we wrote
    \begin{align*}
	\xi_{h,m}(k) = \sum_{j \in \cB(h,m)} (-1)^j x_j^{k}.
    \end{align*}
    We now choose the sets $\cJ_l$ occurring in \eqref{rankcond} according to \eqref{diag} as
    \begin{align*}
    		\cJ_l = \{j(h,m): 1 \le m \le \mu_h-\mu_{h+1}, l \le h \le t\},
    \end{align*}
    where $j(h,m) \in \cB_{h,m}$ for all $h$ and $m$.
    Let $\cJ = \cJ_{1}$, write $C_{l,n}$ for the $(\mu_l \times M)$-matrix defined by
    \begin{align}\label{Cdef}
			C_{l,n} = (c_{ij})_{i \in \cI_{l,n}, j \in \cJ} \qquad (1 \le n \le \nu(l), \ 1 \le l \le t),
    \end{align}
    and let
    \begin{align*}
	\X_{l,n} = \big(\xi_{h,m}(k_{l,n}) \big)_{1 \le m \le \mu_h-\mu_{h+1}, l \le h \le t} \in \Z^{\mu_l}.
    \end{align*}
    Then the system \eqref{eq1} can be written more compactly as
    \begin{align*}
	C_{l,n} \X_{l,n} = {\boldsymbol 0} \quad (1 \le n \le \nu(l), 1 \le l \le t).
    \end{align*}

    We prove the statement by induction. Consider the case $l=1$.   In view of the nonsingularity condition \eqref{rankcond}, we have $\det C_{1,n} \neq 0$ for $1 \le n \le \nu(1)$, and it follows that the equations
    \begin{align*}
	C_{1,1} \X_{1,1} = \dots = C_{1,\nu(1)} \X_{1,\nu(1)} = \bm 0
    \end{align*}
    are satisfied if and only if
    \begin{align}\label{decoup}
	\X_{1,1} = \dots =  \X_{1,\nu(1)} = \bm 0.
    \end{align}
    Consider now those equations within \eqref{decoup} that correspond to $h=1$.  On recalling that $|\cB(1,m)|=2u_1$, we see that this subsystem consists of $\mu_1-\mu_2$ copies of the system
    \begin{align*}
	  \sum_{j=1}^{2u_1} (-1)^j x_j^{k_{1,n}} = 0 \quad (1 \le n \le \nu(1)),
    \end{align*}
    whose solutions are counted by the mean value $J_{u_1,\k_1}(\cA)$. It follows that the total number of solutions of the subsystem corresponding to $h=1$ is given by $(J_{u_1,\k_1}(\cA))^{\mu_1-\mu_2}$.


Suppose now that for some $l$ with $2 \le l \le t$ the systems
    \begin{align}\label{decoup2}
	\X_{h,1} = \dots =  \X_{h,\nu(h)} = \bm 0 \quad (1 \le h \le l-1),
    \end{align}
    have been solved, so that all variables $x_{j}$ with $j \in \cB_{h,m}, 1 \le m \le \mu_h-\mu_{h+1}, 1 \le h \le l-1$ are determined.
    This fixes the values of $\xi_{h,m}(k_{l',n})$ for all $1 \le m \le \mu_h- \mu_{h+1}$, $1 \le h \le l-1$ for all degrees $k_{l',n}$ with $1 \le n \le \nu(l')$, $l \le l' \le t$. We now seek to solve the subsystem associated to the degrees $k_{l,1}, \dots, k_{l, \nu(l)}$. Upon writing
    \begin{align*}
	\a_{l,n} = \bigl(\xi_{h,m}(k_{l,n})\bigr)_{1 \le m \le \mu_h-\mu_{h+1}, 1 \le h \le l-1} \in \Z^{M-\mu_l}
    \end{align*}
    for the vector of variables already determined, then the system is of the shape $C_{l,n} \ZZ_{l,n} = \bm 0$ for $1 \le n \le \nu(l)$, where $\ZZ_{l,n} = (\a_{l,n}; \X_{l,n})$. The nonsingularity condition implies that $C_{l,n} = [A_{l,n} | B_{l,n}]$, where $B_{l,n}$ is a $(\mu_l \times \mu_l)$-matrix with $\det (B_{l,n})  \neq 0$ for $1 \le n \le \nu(l)$.  Hence the system in question is equivalent to the system
    \begin{align}\label{eqh}
	B_{l,n}  \X_{l,n}  + A_{l,n} \a_{l,n} = \bm 0 \quad (1 \le n \le \nu(l)).
    \end{align}
    Write $\p = -( A_{l,n} \a_{l,n})_{1 \le n \le \nu(l)} \in \Z^{\rho_l}$ and $\A_l$ for the vector comprising those components of $\A$ corresponding to the set $\cI_l$.  Further, write $\k(l)=(k_{l,1}, \dots, k_{l,\nu(l)})$ and set $f=f_{\k(l)}$.  Then the number of solutions of the system \eqref{eqh} is given by
    \begin{align*}
	&\int_{[0,1)^{\rho_l}} \prod_{h=l}^t \prod_{m=1}^{\mu_h-\mu_{h+1}} |f(\G_{j(h,m)} ;\cA )|^{2u_h} e(\A_l \cdot \p) \, \dr \A_l \\
	&\qquad \qquad \le   \int_{[0,1)^{\rho_l}} \prod_{h=l}^t \prod_{m=1}^{\mu_h-\mu_{h+1}} |f(\G_{j(h,m)} ;\cA )|^{2u_h} \, \dr \A_l,
    \end{align*}
    and here the latter integral counts solutions of the system
    \begin{align*}
	B_{l,n}  \X_{l,n}= \bm 0 \quad (1 \le n \le \nu(l)).
      \end{align*}
    Since the non-singularity condition implies that $\det(B_{l,n}) \ne 0$, we therefore deduce that the number of solutions to \eqref{eqh} is bounded above by the number of solutions of the system
    \begin{align*}
	\X_{l,n} = \bm 0 \quad (1 \le n \le \nu(l)),
    \end{align*}
    and the contribution stemming from the case $h=l$ can be interpreted as $\mu_l-\mu_{l+1}$ copies of the system
    \begin{align*}
	\sum_{j =1}^{2u_l} (-1)^j x_j^{k_{l,n}} = 0 \quad (1 \le n \le \nu(l)).
    \end{align*}
    Combining this with (\ref{decoup2}), we find that the number of choices for the variables in each of the blocks $\cB(l,m)$ with $1 \le m \le \mu_l-\mu_{l+1}$ is bounded above by the mean value $J_{u_l,\k_l}(\cA)$.
    It now follows by induction that
    \begin{align*}
	I_{\u,\k,\M}(\cA) \ll \prod_{l=1}^t  (J_{u_l,\k_l}(\cA))^{\mu_{l}-\mu_{l+1}},
    \end{align*}
    and this completes the proof of the theorem.
\end{proof}

\section{The minor arcs}

We now describe our Hardy-Littlewood dissection.  For the purpose of the very general Theorem~\ref{mainthm} we can afford to economize on effort by working exclusively with a narrow set of major arcs.  The weakness of the ensuing minor arc estimates is of little consequence to the quality of our bounds, and we avoid pruning arguments.

We take $X \le P$ to be a parameter tending to infinity with $P$.  Define the major arc
\begin{align*}
    \fM(q,\a;X) = \{\A \in [0,1)^r: |q\alpha_i-a_i| \le XP^{-d_i}, 1\le i \le r\},
\end{align*}
and write $\fM(X)$ for the union of all $\fM(q,\a;X)$ with $1 \le \a \le q$, $(q,\a)=1$, and $1 \le q \leqslant X$.  We then write $\fm(X) = [0,1)^r \setmin \fM(X)$ for the minor arcs.

We establish a Weyl-type estimate by exploiting the non-singularity condition for an $M$-tuple of exponential sums.
\begin{lemma} \label{weylest}
    Suppose that $\A \in \fm(X)$. Then there exists $\sigma > 0$ such that for each $M$-tuple $(j_1, \dots, j_M)$ of distinct indices there exists an index $j_i$ for which one has
    \begin{align*}
	|f_{j_i}(\A; [1,P])| \le P X^{-\sigma}.
    \end{align*}
\end{lemma}

\begin{proof}
     Fix $\j$ as in the statement of the lemma, let $\sigma < 1/(2k)$, and suppose that for some $\A \in [0,1)^r$ one has $|f_{j_i}(\A; [1,P])| \geqslant PX^{-\sigma}$ for each $i=1, \dots, M$. Then \cite{P:02ineq3}, Lemma~2.4, implies that there exists $q \ll X^{2k\sigma}$ for which
    \begin{equation} \label{gammapprox}
	\lVert q\gamma_{j_i,l,n} \rVert \ll X^{2k\sigma} P^{-k_{l,n}} \quad (1 \le n \le \nu(l), 1 \le l \le t, 1 \le i \le M).
    \end{equation}
    For ease of reference to the coordinate transform matrices defined in the previous section, we find it convenient to partition the indices as in (\ref{partition}), with $j_1, \dots, j_M$ occurring in distinct blocks.
    Thus we write $\j = (j_i)_{1 \le i \le M} = (j(m,h))_{1\le m \le \mu_h-\mu_{h+1},  1 \le h \le t}$, and for each $l$ and $n$ write
    \begin{align*}
	\G_{l,n}^{*} = \Big((\gamma_{j(h,m),l,n})_{\substack{1 \le m \le \mu_{h} - \mu_{h+1} \\ l \le h \le t }} \Big)^T \qquad     \mbox{and} \qquad \A_{l,n} = (\alpha_{r_{l,n-1}+1}, \dots, \alpha_{r_{l,n}})^T.
    \end{align*}
    We also write $\G_{l,n}$ for the extension of $\G_{l,n}^*$ to all $1 \le h \le t$.  Then the relations \eqref{gamalp} give
    \begin{align*}
	\G_{l,n} = C_{l,n}^T \A_{l,n} \quad ( 1 \le n \le \nu(l), 1 \le l \le t),
    \end{align*}
    where $C_{l,n}=[A_{l,n}|B_{l,n}]$ is the $\mu_l \times M$ coefficient matrix defined in \eqref{Cdef}. It follows from \eqref{rankcond} that $\det(B_{l,n}) \neq 0$ and hence that
    \begin{align*}
	\A_{l,n} = (B_{l,n}^T)^{-1} \G^{*}_{l,n} \quad (1 \le n \le \nu(l), 1\le l \le t).
    \end{align*}
    Thus for each $i$ with $r_{l, n-1}+1 \le i \le r_{l,n}$, one has
    \begin{align*}
	\alpha_i = \sum_{h=l}^t \sum_{m=1}^{\mu_h-\mu_{h+1}} b_{j(h,m),i} \gamma_{j(h,m),l,n},
    \end{align*}
    where the $ b_{j(h,m),i}$ are entries of the matrix $(B_{l,n}^T)^{-1}$ whose moduli are hence bounded above by some absolute constant.
    It follows from \eqref{gammapprox} that
    \begin{align*}
	\lVert q\alpha_i \rVert \le \sum_{h=l}^t \sum_{m=1}^{\mu_l-\mu_{l+1}} |b_{j(h,m),i}| \lVert q\gamma_{j(h,m),l,n} \rVert \ll X^{2k\sigma} P^{-k_{l,n}}.
    \end{align*}
    We therefore deduce that $\A \in \fM(X)$ for $X$ sufficiently large, and the result follows.
\end{proof}

We now complete the analysis of the minor arcs for Theorem~\ref{mainthm}. For case \eqref{asymp}, we set $s =2s_0+1$, write $f_j(\A) = f_j(\A; [1,P])$, and set \begin{equation}
\label{tNdef}
    \tN_{s,\k,\M}(\fB) = \int_{\fB} \prod_{j=1}^s  f_j(\A) \, \dr \A.
    \end{equation}
For $j=1, \dots, M$ and $\sigma > 0$, let $\fm^{(j)}$ denote the set of $\A \in [0,1)^r$ for which $|f_j(\A)| \le PX^{-\sigma}$. For a given index $j$, we
partition the remaining $2s_0$ indices into blocks $\cB_{h,m}$ with $|\cB_{h,m}| = 2u_h$, where $u_h$ is as in Corollary \ref{mvtcor} with $\cA=[1,P]$, so that
\begin{align}\label{cleverweyl}
    \tN_{s,\k,\M}(\fm^{(j)}) \ll P X^{-\sigma} I_{\u,\k,\M}([1,P]).
\end{align}
Lemma~\ref{weylest} ensures that there exists $\sigma$ for which $\fm \subseteq \fm^{(1)} \cup \cdots \cup \fm^{(M)}$,
and it follows from Corollary \ref{mvtcor} that whenever $X$ is a small power of $P$ and $\ep$ is small enough one has
\begin{align*}
    \tN_{s,\k,\M}(\fm) \ll P X^{-\sigma} P^{2s_0-K+\ep} \ll P^{s-K}X^{ - \sigma/2}.
\end{align*}
In case \eqref{exist} we set $s=2s_0+M$ and partition the indices $j = M+1, \dots, s$ as before, but with the block sizes determined by Corollary \ref{mvtcor} with $\cA=\cA(P,R)$.  Here we write
\begin{align*}
    N_{s,\k,\M}(\fB) = \int_{\fB} \prod_{i=1}^M f_i(\A) \prod_{h=1}^t \prod_{m=1}^{\mu_h-\mu_{h+1}} \prod_{j \in \cB_{h,m}} g_j(\A) \, \dr \A,
\end{align*}
where we suppose that $f_i(\A) = f_i(\A; [1,P])$ for $1 \le i \le M$ and $g_j(\A) = f_j(\A,\cA(P,R))$ for $M+1 \le j \le s$. Then it follows from Lemma~\ref{weylest} that
\begin{align*}
    N_{s,\k,\M}(\fm) \ll P^{M} X^{-\sigma} I_{\u,\k,\M}(\cA(P,R)),
\end{align*}
and when $\ep$ is sufficiently small an application of Corollary \ref{mvtcor} delivers the bound
\begin{equation}\label{minbd}
N_{s,\k,\M}(\fm) \ll P^{s-K} X^{-\sigma/2}.
\end{equation}
This completes the analysis of the minor arcs in the setting of Theorem~\ref{mainthm}.

\section{The major arcs}

We complete the proof of Theorem~\ref{mainthm} by obtaining the expected contribution from our thin set of major arcs. Although the analysis is in principle relatively routine, the combination of repeated and differing degrees requires us to exercise some care in adapting existing approaches.  As with our minor arc estimates, we make critical use of the non-singularity condition to extract non-singular sub-matrices of coefficients.

Set $X=(\log P)^{1/(6r)}$ if $\cA=\cA(P,R)$ or $X = P^{1/(6r)}$ when $\cA = [1,P]$, and consider the slightly expanded major arcs
\begin{align*}
    \fN(X) = \bigcup_{q=1}^X \bigcup_{\substack{\a =1  \\ (q, \a)=1}}^q \fN(q,\a;X),
\end{align*}
where $\fN(q,\a;X)$ is given by the set of all $\A  \in [0,1)^r$ satisfying
\begin{align*}
     |\alpha_{l,n}-q^{-1}a_{l,n}| \le XP^{-k_{l,n}} \quad (1 \le n \le \nu(l), 1 \le l \le t).
\end{align*}
Then $\fn (X)=[0,1)^r\setminus \fN(X) \subseteq \fm(X)$, and the work of the previous section implies that the contribution of the minor arcs is negligible compared to the expected main term.

We write
\begin{align*}
    S(q,\a) = \sum_{x=1}^q e\left((a_{1,1} x^{k_{1,1}} + \dots + a_{t,\nu(t)} x^{k_{t,\nu(t)}})/q\right)
\end{align*}
and recall that the argument of \cite{V:HL}, Theorem~7.1 (see also \cite{P:02ineq3}, equation (2.2)) gives
\begin{align}\label{Sqaest}
    S(q,\a) \ll (q,\a)^{1/k} q^{1-1/k+\ep}.
\end{align}

Further define $\omega=0$ if $\cA=[1,P]$ and $\omega=1$ when $\cA = \cA(P,R)$,
and set
\begin{align*}
    v(\B;P) = \int_{\omega R}^P \rho\biggl(\frac{\log z}{\log R}\biggr)^{\omega} e(\beta_{1,1}z^{k_{1,1}} + \dots + \beta_{t,\nu(t)}z^{k_{t,\nu(t)}}) \, \dr z,
\end{align*}
where $\rho$ denotes Dickman's function.
We recall from the arguments of \cite[Theorem~7.3]{V:HL} and \cite[Lemma~8.6]{W:91sae2} (see also \cite{P:02ineq3}, equations (2.3) and (2.4)) the estimate
\begin{align}\label{vest}
    v(\B; P) \ll P\bigg(1+\sum_{l=1}^t \sum_{n=1}^{\nu(l)}|\beta_{l,n}|P^{k_{l,n}}\bigg)^{-1/k}.
\end{align}
It then follows easily that when $\A  = \a/q + \B \in \fN(q,\a,X) \subseteq \fN(X)$, one has
\begin{align*}
    f_j(\A) = q^{-1} S(q,\LL_j) v(\D_j;P) + O(X^2 P^{\omega} (\log P)^{-\omega}),
\end{align*}
where
\begin{align}\label{lama}
    \Lambda_{j,l,n} = \sum_{i=r_{l,n-1}+1}^{r_{l,n}} c_{ij} a_i   \quad (1 \le j \le s, 1 \le n \le \nu(l), 1 \le l \le t)
\end{align}
and
\begin{align}\label{delbet}
    \delta_{j,l,n} = \sum_{i=r_{l,n-1}+1}^{r_{l,n}} c_{ij} \beta_i  \quad (1 \le j \le s, 1 \le n \le \nu(l), 1 \le l \le t),
\end{align}
so that $\D=\G-\LL/q$.  We write $S_j(q,\a)=S(q,\LL_j)$ and $v_j(\B; P)=v(\D_j; P)$, and define
\begin{align*}
    \fS(X) = \sum_{q \leqslant X} \sum_{\substack{1 \leqslant \a \leqslant q \\ (q,\a)=1}} \prod_{j=1}^s q^{-1} S_j(q,\a)
 \quad\text{and}\quad
    \fJ(X) = \int_{\cI(P,X)} \prod_{j=1}^s v_j(\B;P) \, \dr \B,
\end{align*}
with
\begin{align*}
    \cI(P,X) = \bigtimes_{l=1}^t \bigtimes_{n=1}^{\nu(l)}  [-XP^{-k_{l,n}}, XP^{-k_{l,n}}]^{\mu_l}.
\end{align*}
Then since $\vol \fN(X) \ll X^{2r+1}P^{-K}$, one finds that
\begin{align}\label{majorasymp}
\int_{\fN(X)} f_1(\A) \cdots f_s(\A) \, \dr \A = \fS(X) \fJ(X) + O(P^{s-K}(\log P)^{-\nu})
\end{align}
for some $\nu > 0$.

We now show that one can complete the singular series and singular integral as usual by defining, for each fixed $P$,
\begin{align} \label{completions}
    \fS = \lim_{Y \to \infty} \fS(Y) \quad \mbox{and} \quad \fJ=\lim_{Y \to \infty} \fJ(Y).
\end{align}

We first complete the singular series.
Write
\begin{align*}
  A(q) = q^{-s}\sum_{\substack{1 \leqslant \a \leqslant q \\ (q,\a)=1}} \prod_{j=1}^s  S_j(q,\a),
\end{align*}
and note that $A(q)$ is multiplicative in $q$, whence the singular series, if convergent, can be written as
\begin{align}\label{ss-conv}
 \fS = \prod_p \sum_{i=0}^\infty A(p^i).
\end{align}
We show that the product in \eqref{ss-conv} converges.

\begin{lemma}\label{ss}
		Suppose the system $(\ref{sys})$ is highly non-singular with $s>2s_0$, where $s_0$ is given by \eqref{s0def} with
		\begin{align*}
		    u_h \ge \frac{k}{2}(1+\varpi_h) \qquad (1 \le h \le t).
		\end{align*}
		Then the singular series is absolutely convergent, and one has $\fS - \fS(X) \ll X^{-\delta}$ for some $\delta > 0$.
\end{lemma}

\begin{proof}
    We partition the indices as in (\ref{partition}) and
    let $v_h = 2u_h + (s-2s_0)/M > 2u_h$ for $1 \le h \le t$. Then one has
    \begin{align*}
	\sum_{h=1}^t v_h(\mu_h-\mu_{h+1}) = s,
    \end{align*}
    and hence by \eqref{trivineq} there exists $\j \in \cB_{1,1} \times \dots \times \cB_{t,\mu_t}$ with the property that
    \begin{align*}
	\sum_{\substack{1 \leqslant \a \leqslant q \\ (q,\a)=1}} \prod_{j=1}^{s}  S_j(q,\a) \ll  \sum_{\substack{1 \leqslant \a \leqslant q \\ (q,\a)=1}} \prod_{h=1}^{t} \prod_{m=1}^{\mu_h-\mu_{h+1}} |S_{j(h,m)}(q,\a)|^{v_h}.
    \end{align*}
    We now apply the change of variables \eqref{lama}. Thus, on writing $\a_{l,n}=(a_{r_{l,n-1}+1}, \dots, a_{r_{l,n}})^T$ and $\LL^*_{l,n} = (\Lambda_{j(h,m),l,n})^T_{h,m}$ with $1 \le m \le \mu_{h}-\mu_{h+1}$ and $l \le h \le t$, we obtain $\LL^*_{l,n} = B_{l,n}^T \a_{l,n}$, where the matrix $B_{l,n}$ is as in (\ref{eqh}). In particular, one has $\det B_{l,n} \neq 0$ for all $1 \le n \le \nu(l)$ and $1 \le l \le t$.  As a result, the remaining coefficients $\Lambda_{j(h,m),l,n}$ with $1 \le m \le \mu_h-\mu_{h+1}$ and $1 \le h \le l-1$ may be expressed as linear combinations of those $\Lambda_{j(h,m),l,n}$ having $h \ge l$. Then on writing
    \begin{align*}
	\LL_{j(h,m)}=(\Lambda_{j(h,m),l,n})_{\substack{1 \le n \le \nu(l) \\ 1 \le l \le t}} \qquad \text{and} \qquad \LL = (\LL_{j(h,m)})_{\substack{1 \le m \le \mu_h-\mu_{h+1} \\ 1 \le h \le t}},
    \end{align*}
    the invertibility of the transformation further implies that the coefficients of $\LL$ occurring in these relations satisfy $(q, \LL) \ll 1$ whenever $(q,\a)=1$. Hence there exist constants $C, C'$ for which
    \begin{align*}
	\sum_{\substack{1 \leqslant \a \leqslant q \\ (q,\a)=1}} \prod_{h=1}^{t} \prod_{m=1}^{\mu_h-\mu_{h+1}} |S_{j(h,m)}(q,\a)|^{v_h} &\ll  \sum_{\substack{|\LL| \leqslant Cq \\ (q,\LL) \le C'}} \prod_{h=1}^t\prod_{m=1}^{\mu_h-\mu_{h+1}}  |S(q,\LL_{j(h,m)})|^{v_h}.
    \end{align*}
    It follows from \eqref{Sqaest} that
    \begin{align*}
	A(p^i) &\ll p^{-is} \sum_{\substack{|\LL| \leqslant C p^i\\ (p^i,\LL) \le C'}} \prod_{h=1}^t\prod_{m=1}^{\mu_h-\mu_{h+1}}  |S(p^i,\LL_{j(h,m)})|^{v_h} \\
	& \ll p^{-is/k+\ep} \sum_{\substack{|\LL| \leqslant C p^i\\ (p^i,\LL) \le C'}} \prod_{h=1}^t\prod_{m=1}^{\mu_h-\mu_{h+1}} (p^i,\LL_{j(h,m)})^{v_h/k}.
    \end{align*}
    Let $\kappa(p)$ denote the largest integer satisfying $p^{\kappa(p)} \le C'$ and define $e_{h,m}$ via $(p^i,\LL_{j(h,m)}) = p^{e_{h,m}}$.  Then one has
    \begin{align*}
	A(p^i)  \ll p^{-is/k+\ep} \sum_{\e} \bigg(\prod_{h=1}^t\prod_{m=1}^{\mu_h-\mu_{h+1}}  p^{e_{h,m} v_h/k }\bigg) \Xi(p^i, \e),
    \end{align*}
    where the sum is over all $0 \le e_{h,m} \le i$ with the condition that $e_{h,m}\le \kappa(p)$ for at least one pair of indices $(h,m)$, and $\Xi(p^i, \e)$ denotes the number of $\LL \le Cp^i$ satisfying $(p^i,\LL_{j(h,m)}) = p^{e_{h,m}}$ for every $h$ and $m$. Recalling that for $h < l$ the coefficients $\Lambda_{j(h,m),l,n}$ are linearly dependent on $(\Lambda_{j(h,m),l,n})_{h \ge l}$, it suffices to determine the number of choices for those coefficients where $h \ge l$, in which case the number of choices for any given $\Lambda_{j(h,m),l,n}$ is certainly bounded above by $p^{i-e_{h,m}}$. It follows that
    \begin{align*}
	\Xi(p^i, \e) &\ll \prod_{l=1}^t \prod_{h=l}^t \prod_{m=1}^{\mu_h-\mu_{h+1}} \left( p^{i- e_{h,m} }\right)^{\nu(l)} \ll p^{ir} \prod_{h=1}^t\prod_{m=1}^{\mu_h-\mu_{h+1}} p^{- e_{h,m}  \varpi_h },
    \end{align*}
    where we used \eqref{r} and \eqref{def-varpi}.
    Thus altogether we have the estimate
    \begin{align*}
	A(p^i) \ll p^{-is/k+ir + \ep} \sum_{\e} \prod_{h=1}^t\prod_{m=1}^{\mu_h-\mu_{h+1}}  p^{e_{h,m} (v_h/k -\varpi_h)}.
    \end{align*}
    Observe that the sum over $\e$ essentially amounts to a divisor function with the additional constraint that at least one of the $e_{h,m}$ must be bounded above by $\kappa(p)$ in order to satisfy coprimality.  Thus after executing the summation one finds that $A(p^i) \ll p^{-i(s/k-r)+\xi+\ep}$, where
    \begin{align*}
	\xi &\le i \Big(\sum_{h=1}^t (\mu_h-\mu_{h+1}) (v_h/k - \varpi_h)\Big) - (i-\kappa(p))\min_{h}(v_h/k - \varpi_h)\\
	  & = i(s/k - r) - (i-\kappa(p))\min_{h}(v_h/k - \varpi_h ),
    \end{align*}
    and since $p^{\kappa(p)}$ is bounded by an absolute constant, we obtain
    \begin{align}\label{A-bd}
	A(p^i)&  \ll p^{-i\min_{h}(v_h/k - \varpi_h) + \ep}.
    \end{align}
    On recalling that
    \begin{align*}
	u_h \ge \frac{k}{2}(1+\varpi_h)\qquad (1 \le h \le t),
    \end{align*}
    the fact that $v_h >2u_h$ implies that one has $A(p^i) \ll p^{i(-1-\tau)}$ for some $\tau > 0$, uniformly for $i \in \N$.  It follows that
    \begin{align*}
	\fS = \prod_p \bigg(1 + \sum_{i=1}^{\infty} A(p^i)\bigg) \le \prod_p(1+ cp^{-1-\tau})
    \end{align*}
    for some constant $c>0$, and
    this establishes the convergence of the singular series in the setting of Theorem~\ref{mainthm}. The second assertion of the lemma follows immediately.
\end{proof}

We now turn to the completion of the singular integral.
\begin{lemma}\label{si}
    Suppose the system $(\ref{sys})$ is highly non-singular with $s>2s_0$, where $s_0$ is given by \eqref{s0def} with
    \begin{align*}
	u_h > \tfrac{1}{2} k \varpi_h \quad (1 \le h \le t).
    \end{align*}
    Then one has
    \begin{align*}
	\fJ - \fJ(X) \ll P^{s-K}X^{-1}.
    \end{align*}
\end{lemma}

\begin{proof}
    Firstly, observe that by a change of variables one has
    \begin{align*}
	\fJ(X) &\ll P^{s-K} \int_{[-X,X]^r} \prod_{j=1}^s v_j(\B; 1) \, \dr \B.
    \end{align*}
    We now partition the indices as in (\ref{partition}).
    By \eqref{trivineq}, there exists $\j \in \cB_{1,1} \times \dots \times \cB_{t,\mu_t}$ with the property that
    \begin{align*}
	\fJ-\fJ(X) \ll P^{s-K}\int_{\cR} \prod_{h=1}^t\prod_{m=1}^{\mu_h-\mu_{h+1}} |v_{j(h,m)}(\B;1)|^{2u_h} \, \dr \B,
    \end{align*}
    where the set $\cR$ contains all vectors $\B$ satisfying
    \begin{align*}
	\displaystyle{ \max_{1 \le l \le t} \max_{1 \le n \le \nu(l) } \max_{i \in \cI_{l,n}} }\{ |\beta_{i}|\} > X.
    \end{align*}
    We make the change of variables \eqref{delbet}, and write
    \begin{align*}
	\B_{l,n}=(\beta_{r_{l,n-1}+1}, \dots, \beta_{r_{l,n}})^T \qquad \mbox{and} \qquad \D^*_{l,n} = \Big((\delta_{j(h,m),l,n})_{\substack{1\le m \le \mu_h-\mu_{h+1}\\l \le h \le t}}\Big)^T.
    \end{align*}
    We then find as above that $\D^*_{l,n} = (B_{l,n})^T \B_{l,n}$ and $\det B_{l,n} \neq 0$ $( 1 \le n \le \nu(l), 1 \leqslant l \leqslant t)$.
    Hence the remaining coordinates $\delta_{j(h,m),l,n}$ with $1 \le m \le \mu_h-\mu_{h+1}$ and $1 \le h \le l-1$ are linear combinations of those $\delta_{j(h,m),l,n}$ having $1 \le m \le \mu_h-\mu_{h+1}$ and $l \le h \le t$, and the non-singularity of the coordinate transform implies further that
    \begin{align*}
       \max_{1 \le l \le t} \max_{1 \le n \le \nu(l)} \max_{ l \le h \le t}\max_{1 \le m \le \mu_h - \mu_{h+1}}|\delta_{j(h,m),n,l}| \gg X
    \end{align*}
    whenever $\B \in \cR$. We will write $\D^{(i)}$ for the vector comprising all $\delta_{j(h,m),l,n}$ with $i \le l \le h \le t$, $1 \le n \le \nu(l)$, and $1 \le m \le \mu_h-\mu_{h+1}$.  After integrating with respect to those components of $\D=\D^{(1)}$ having $l=1$, one obtains from (\ref{vest}) that
    \begin{align*}
	\fJ  &\ll \int_{\R^{r}} \prod_{h=1}^t \prod_{m=1}^{\mu_h-\mu_{h+1}} \bigg(1+\sum_{l=1}^t \sum_{n=1}^{\nu(l)}|\delta_{j(h,m),l,n}|\bigg)^{-\frac{2u_h}{k}} \, \dr \D \\
	& \ll \int_{\R^{r-r_1}} \prod_{h=1}^t \prod_{m=1}^{\mu_h-\mu_{h+1}} \bigg(1+\sum_{l=2}^t \sum_{n=1}^{\nu(l)}|\delta_{j(h,m),l,n}|\bigg)^{-\frac{2u_h}{k} + \nu(1)} \, \dr \D^{(2)},
    \end{align*}
    provided that $u_h > \frac{k}{2}\nu(1)$ for all $h$. The resulting integral may be simplified by exploiting the fact that the variables $\delta_{j(1,m),l,n}$ with $l \ge 2$ are linear combinations of the components of $\D^{(2)}$. This implies that
    \begin{align*}
	\fJ &\ll \int_{\R^{r-r_1}} \prod_{h=2}^t \prod_{m=1}^{\mu_h-\mu_{h+1}} \bigg(1+\sum_{l=2}^t \sum_{n=1}^{\nu(l)}|\delta_{j(h,m),l,n}|\bigg)^{-\frac{2u_h}{k} +\nu(1) + (\mu_1-\mu_2)(-\frac{2u_1}{k} +\nu(1))  } \, \dr \D^{(2)}\\
	& \ll \int_{\R^{r-r_1}} \prod_{h=2}^t \prod_{m=1}^{\mu_h-\mu_{h+1}} \bigg(1+\sum_{l=2}^t \sum_{n=1}^{\nu(l)}|\delta_{j(h,m),l,n}|\bigg)^{-\frac{2u_h}{k} +\nu(1) } \, \dr \D^{(2)},
    \end{align*}
    where in the last step we used the assumption $u_1 > \frac{k}{2}\nu(1)$ again to simplify the exponent. We may now iterate the procedure for increasing values of $l$. Thus, provided that $u_h > \frac{k}{2}(\nu(1)+\nu(2))$ for all $h \ge 2$, the same argument yields
    \begin{align*}
	\fJ   &\ll \int_{\R^{r-r_2}} \prod_{h=2}^t \prod_{m=1}^{\mu_h-\mu_{h+1}} \bigg(1+\sum_{l=3}^t \sum_{n=1}^{\nu(l)}|\delta_{j(h,m),l,n}|\bigg)^{-\frac{2u_h}{k} +\nu(1) +\nu(2)  } \, \dr \D^{(3)}\\
	& \ll \int_{\R^{r-r_2}} \prod_{h=3}^t \prod_{m=1}^{\mu_h-\mu_{h+1}} \bigg(1+\sum_{l=3}^t \sum_{n=1}^{\nu(l)}|\delta_{j(h,m),l,n}|\bigg)^{-\frac{2u_h}{k} +\nu(1) + \nu(2)  } \, \dr \D^{(3)},
    \end{align*}
    and after $t$ iterations we obtain convergence if $u_h > \frac{1}{2}k \varpi_h \ (1 \le h \le t)$.
    Furthermore, it is clear that under the same condition one has
    \begin{align}\label{J-bd}
	\fJ-\fJ(X) \ll P^{s-K}\int_{\cR} \prod_{h=1}^t\prod_{m=1}^{\mu_h-\mu_{h+1}} |v_{j(h,m)}(\B,1)|^{2u_h} \, \dr \B \ll  P^{s-K} X^{-\nu}
    \end{align}
    for some $\nu > 0$.
\end{proof}

A coordinate transform now shows that $\fJ = P^{s-K} \chi_{\infty}$ with
\begin{align*}
    \chi_\infty = \int_{\R^r} \int_{[0,1]^s} e\Big(\sum_{i=1}^r \beta_i \Theta_i(\ZZ)\Big) \, \dr \ZZ  \, \dr \D,
\end{align*}
where where $\Theta_i(\x) = c_{i1}x_1^{d_i} + \dots + c_{is}x_s^{d_i}$, and it follows from Lemma~\ref{si} that $\chi_\infty$ is a finite constant. Furthermore, the argument of \cite[Lemma~7.4]{P:02pae} is easily adapted to prove that, under the conditions of Lemma~\ref{si}, this constant is positive whenever the system \eqref{sys} possesses a non-singular real solution in the positive unit hypercube.

%
Also, a standard argument yields
\begin{align*}
    \chi_p = \sum_{i=0}^\infty A(p^i) = \lim_{i \to \infty} p^{-i(s-r)} M(p^i),
\end{align*}
where $M(p^i)$ denotes the number of solutions of the congruences to the modulus $p^i$ which correspond to the equations \eqref{sys}. It follows from \eqref{A-bd} that $\chi_p = 1+O(p^{-1}) \ge \tfrac{1}{2}$ for $p$ sufficiently large, and for small primes one uses Hensel's lemma to deduce that $\chi_p>0$ if the system \eqref{sys} possesses a non-singular $p$-adic solution.  The proof of Theorem~\ref{mainthm} is now complete on recalling (\ref{minbd}), (\ref{majorasymp}), Lemma~\ref{ss}, and Lemma~\ref{si}, and the constant in \eqref{N-asymp} is given by $ C = \chi_\infty \prod_{p}\chi_p$.

\section{Auxiliary estimates for systems of cubics and quadratics}

The proof of Theorem~\ref{quadcubthm} requires a more careful treatment. In this section we collect a number of auxiliary results that will be of use when we complete the proof in the final section.
Here the system is given by
\begin{align}\label{sys23}
  c_{i1} x_1^3 + \dots + c_{is} x_s^3 &= 0 \quad (1 \le i \le r_C),\nonumber \\
  d_{i1} x_1^2 + \dots + d_{is}x_s^2 &= 0 \quad (1 \le i \le r_Q),
\end{align}
whence the relation \eqref{gamalp} reduces to
\begin{align}\label{gamalp2}
  \gamma_{2,j} = \sum_{i=1}^{r_Q} d_{ij}\alpha_{2,i} \quad \mbox{and} \quad \gamma_{3,j} &= \sum_{i=1}^{r_C}c_{ij}\alpha_{3,i},
\end{align}
and the exponential sum takes the shape
\begin{align*}
  f_{j}(\A; \cA) =  \sum_{x \in \cA} e( \gamma_{3,j} x^3 + \gamma_{2,j} x^2) = f(\G_j).
\end{align*}
Note in particular that, since the respective ranks of the coefficient matrices $(c_{ij})$ and $(d_{ij})$ are $r_C$ and $r_Q$, only $r=r_Q+r_C$ of the $2s$ entries of $\G$ are independent.

For $i \in \{2,3\}$ define
\begin{align*}
  \fM_i(X) = \bigcup_{1 \le q \le X} \{\alpha \in [0,1]: \|q\alpha\| < XP^{-i}\}
\end{align*}
and
\begin{align*}
  \fN_i(X) = \bigcup_{0 \le a < q \le X} \{\alpha \in [0,1]: |\alpha-a/q| < XP^{-i}\},
\end{align*}	
and write $\fM^*(X) = \fM_2(X) \times \fM_3(X)$. The respective complementary sets will be denoted with lower case letters, adorned with the same suffices or asterisks. Furthermore, for $X<Q$ write $\cM_i(Q, X) = \fM_i(Q)\setminus \fN_i(X)$ and $\cM^*(Q, X) = \fM^*(Q) \setminus \fN^*(X)$.

\begin{lemma}\label{aux}
    Suppose that $f(\A) = f(\A; [1,P])$, set $Q=P^{3/4}$, and let $Y$ be a positive number.
    \begin{enumerate}[(i)]
    \item For any $u > 2$, one has
      \begin{align*}
	\sup_{\alpha_{k_2} \in \fM_{k_2}(Q)} \int_{\cM_{k_1}(Q, Y)} |f(\A)|^{2u} \dr \alpha_{k_1} &\ll P^{2u-k_1}(Y^{-1}+P^{3/2-u+\ep}).
      \end{align*}
    \item For any $u>7$, one has
      \begin{align*}
	 			\int_{\cM^*(Q, Y)} |f(\A)|^{u} \dr \A &\ll P^{u-5}(Y^{-1} + P^{3-u/2+\ep}).
      \end{align*}
    \end{enumerate}
\end{lemma}

\begin{proof}
    It follows from \cite[Lemma~4.4]{Bak:DI} that for $\A = \a/q+\B  \in \fM^*(Q)$ one has
    \begin{align*}
	f(\A) \ll q^{-1}S(q, \a) v(\B; P) + Q^{2/3+\ep}.
    \end{align*}
    In the case of the second expression we therefore obtain the bound
    \begin{align*}
	\int_{\cM^*(Q, Y)} |f(\A)|^{u} \dr \A \ll \int_{\cM^*(Q,Y)} |q^{-1}S(q, \a) v(\B;P)|^{u} \dr \A + Q^{2u/3+\ep} \vol\fM^*(Q),
    \end{align*}
    and it follows from the argument of Lemma~8.3 (ii) in \cite{W:15sae5} that
    \begin{align*}
	\int_{\cM^*(Q, Y)} |q^{-1}S(q, \a) v(\B;P)|^{u} \dr \A  \ll P^{u-5}Y^{-1}
    \end{align*}
    whenever $u>7$.
    Upon noting that $\vol \fM^*(Q) \ll Q^4P^{-5}$, this establishes the bound claimed in (ii).

    We now consider case (i).  To simplify notation, we write $i=k_1$ and $j=k_2$ for the remainder of the proof.  Analogously to the above argument, for $\alpha_{j} \in \fM_j(Q)$ we have
    \begin{align}
    \label{1dmajor}
	\int_{\cM_{i}(Q, Y)} |f(\A)|^{2u} \dr \alpha_{i} &\ll  \sum_{q=1}^{\infty}\sum_{\substack{a_{i}=1 \\ (\a,q)=1}}^{q} |q^{-1}S(q, \a)|^{2u} \int_{YP^{-i}}^{\infty}  |v(\B;P)|^{2u} \dr \beta_{i} \notag \\
	& \qquad + Q^{4u/3+\ep} \vol \fM_i(Q).
    \end{align}
    Now \eqref{vest}, together with the argument of Lemma~\ref{si}, yields
    \begin{align*}
	\int_{YP^{-i}}^\infty |v(\B;P)|^{2u} \dr \beta_{i} &\ll  P^{2u}\int_{YP^{-i}}^{\infty} (1 + |\beta_{i}|P^{i} + |\beta_{j}|P^{j} )^{-2u/3} \dr \beta_{i} \ll P^{2u-i}Y^{-1}
    \end{align*}
    for all $u > 3/2$. Furthermore, we have
    \begin{align*}
	\sum_{q=1}^{\infty}\sum_{\substack{a_{i}=1 \\ (\a,q)=1}}^{q} q^{-2u}S(q,\a)^{2u} \ll \prod_{p} \sum_{l=0}^\infty A_{[a_{j}]}(p^l),
    \end{align*}
    where
    \begin{align*}
	A_{[a_{j}]}(q) = q^{-2u} \Bigg|\sum_{\substack{a_{i}=1\\(\a,q)=1}}^{q}S(q,\a)^{2u} \Bigg|.
    \end{align*}
    Observe that \eqref{Sqaest} gives $q^{-2u}S(q,\a)^{2u} \ll q^{-2u/3+\ep}$ whenever $(q, \a)=1$, whence
    \begin{align*}
	A_{[a_{j}]}(p^l) \ll \sum_{\substack{a_{i}=1 \\ (\a,p)=1}}^{p^l} p^{-(2/3)u l +\ep} \ll p^{(1-2u/3)l+\ep}.
    \end{align*}
    Furthermore, for $l\in \{1,2\}$ we have the estimate
    \begin{align}\label{Sqaest2}
	S(p^l,\a) \ll (p^l,\a)^{1/2} p^{l/2+\ep}
    \end{align}
    following from \cite{Sch:Eq}, Corollary II.2F and from the argument of the proof of \cite[Theorem~7.1]{V:HL} (see also Lemma~7.1 in \cite{W:15sae5}). We therefore have the bound
    \begin{align*}
	A_{[a_{j}]}(p^l) \ll \sum_{\substack{a_{i}=1 \\ (\a,p)=1}}^{p^l} p^{-u l+\ep} \ll p^{(1-u)l+\ep}
    \end{align*}
    for $l\in \{1,2\}$, and thus altogether
    \begin{align*}
	\sum_{l=0}^{\infty} A_{[a_{j}]}(p^l) &= 1 + O(p^{1-u+\ep} + p^{3-2u+\ep}).
    \end{align*}
    It follows that for some suitable absolute constants $c_1$, $c_2$, $c_3$ and $\delta > 0$ we have
    \begin{align*}
	\prod_{p} \sum_{l=0}^\infty A_{[a_{j}]}(p^l) \ll \prod_{p} (1+ c_1 p^{1-u+\ep} +c_2p^{3-2u+\ep}) \ll \prod_{p}(1+c_3p^{-1-\delta}),
    \end{align*}
    whenever $u>2$ and $\ep$ is small enough.
    The proof is now completed on inserting our estimates into (\ref{1dmajor}), noting that $\vol \fM_i(Q) \ll Q^2P^{-i}$ for $i \in \{2,3\}$, and recalling that $Q=P^{3/4}$.
\end{proof}

In order to establish a suitable pruning lemma for smooth exponential sums we first need an additional auxiliary result.
Let
\begin{align*}
    I(\beta_2,\beta_3) = \int_{\frac12 P}^P e(\beta_2x^2 + \beta_3 x^3) \dr x.
\end{align*}
The following is a modification of Theorem~7.3 of \cite{V:HL}.
\begin{lemma}\label{I-bd}
    We have
    \begin{align*}
	I(\beta_2,\beta_3) \ll P(1 + P^2|\beta_2| + P^3|\beta_3|)^{-1/2}.
    \end{align*}
\end{lemma}
\begin{proof}
    As in the proof of Theorem~7.3 of \cite{V:HL} we observe that the claim is, via a change of variables, equivalent to
    \begin{align*}
	\int_{\frac12}^1 e(\beta_2x^2 + \beta_3 x^3) \dr x \ll (1 + |\beta_2| + |\beta_3|)^{-1/2}.
    \end{align*}
    Let $p(x) = 2\beta_2x + 3\beta_3x^2$. If $\sA$ denotes the set of all $x \in [1/2,1]$ satisfying $|p(x)| \ge (|\beta_2| + |\beta_3|)^{1/2}$, then the contribution from this set is given by
    \begin{align*}
	\int_{\sA}  e(\beta_2x^2 + \beta_3 x^3) \dr x \ll (|\beta_2| + |\beta_3|)^{-1/2}.
    \end{align*}
    It thus remains to bound the contribution of $\sC = [1/2,1] \setminus \sA$. Either $\sC$ is empty, in which case there is nothing to prove, or we can find $\alpha \in [1/2,1]$ with $|p(\alpha)| < (|\beta_2| + |\beta_3|)^{1/2}$. On the other hand, by the triangle inequality we have
    \begin{align*}
	|p(\alpha)|  \ge |2 \beta_2 \alpha|-|3 \beta_3 \alpha^2| \ge |\beta_2|-3|\beta_3|.
    \end{align*}
    In the case when $|\beta_2| \ge 6|\beta_3|$, we thus have $\frac12|\beta_2| \le |p(\alpha)| \le (|\beta_2| + |\beta_3|)^{1/2} \le (\frac76 |\beta_2|)^{1/2}$, so $|\beta_2| \le 14/3$, but for $|\beta_3| \ll |\beta_2| \ll 1$ the claim is trivial. We may therefore assume that $|\beta_2| < 6|\beta_3|$, so that for each $\alpha \in \sC$ one has $|p(\alpha)| \le (7|\beta_3|)^{1/2}$. Since we made the assumption that $\alpha \ge 1/2$, this implies that $\frac12 |2\beta_2+3\beta_3\alpha| \le (7|\beta_3|)^{1/2}$. It follows that the measure of $\sC$ is bounded above by
    \begin{align*}
	\mathrm{vol} \{ 1/2 \le \alpha \le 1: |2\beta_2+3\beta_3\alpha| \le 2(7|\beta_3|)^{1/2}  \} \ll |\beta_3|^{-1/2}.
    \end{align*}
    This establishes the statement.
\end{proof}
More generally, a similar argument can be used to show that for any set of degrees $k_1<\dots< k_t$ one can find some suitable constant $0 < \xi<1$ such that
\begin{align*}
    \int_{\xi P}^P e\bigg(\sum_{j=1}^t\beta_jx^{k_j}\bigg) \dr x \ll P\bigg(1 + \sum_{j=1}^t P^{k_j}|\beta_j|\bigg)^{\! \! -1/t},
\end{align*}
replacing the exponent $1/k_t$ that can be directly inferred from Theorem~7.3 of \cite{V:HL} with the stronger $1/t$.

We are now in a position to establish the main pruning lemma for systems of cubic and quadratic forms, and here we largely follow the treatment devised by Br\"udern and Wooley \cite{BW:07}.
In what follows, we write $g(\A) = f(\A; [\frac12 P, P])$ and $h(\A) = f(\A; \cA)$, where $\cA$ denotes either $[1,P]$ or $\cA(P,R)$.

\begin{lemma}\label{prunesm}
  Let $A \in \Q$ be fixed, and let $Q=P^{3/4}$.
  \begin{enumerate}[(i)]
      \item\label{bw07lem}
	  For any $\delta>0$ one has the relation
	  \begin{align*}
	      \sup_{\lambda, \mu \in \R}\sup_{\alpha_2 \in \fM_2(Q)}\int_{\cM_3(Q,X)}|g(\alpha_3,\alpha_2)^{2+\delta} h(A\alpha_3+\lambda,\mu)^2| \dr \alpha_3 \ll P^{1+\delta}X^{-\delta/2}.
	  \end{align*}
      \item\label{bw07lem2} Additionally, one has
	  \begin{align*}
	      \sup_{\lambda, \mu \in \R}\sup_{\alpha_2 \in \fM_2(Q)}\int_{\cM_3(Q,X)}|g(\alpha_3,\alpha_2) h(A\alpha_3+\lambda, \mu)^6| \dr \alpha_3 \ll P^{4}X^{-1/6}.
	  \end{align*}
  \end{enumerate}
\end{lemma}

\begin{proof}
    We first show \eqref{bw07lem}. This follows almost directly from the argument of the proof of \cite[Lemma~9]{BW:07}. If $A = B/S$, where $B \in \Z$ and $S \in \N$, then by a change of variables one has
    \begin{align*}
	&\int_{\cM_3(Q,X)}|g(\alpha_3,\alpha_2)^{2+\delta} h(A\alpha_3+\lambda, \mu)^2| \dr \alpha_3 \\
	&\quad = S \int_{S^{-1} \cM_3(Q,X)}|g(S\alpha_3,\alpha_2)^{2+\delta} h(B\alpha_3+\lambda, \mu)^2| \dr \alpha_3.
    \end{align*}
    Let $\kappa$ denote the multiplicative function defined by
    \begin{align*}
	\kappa(p^i) = \begin{cases}
		    p^{-i/3} & i \ge 3, \\ p^{-i/2} & i \in \{1,2\}.
                  \end{cases}
    \end{align*}
    Then as a consequence of Lemma~4.4 in \cite{Bak:DI}, equations \eqref{Sqaest} and \eqref{Sqaest2}, and Lemma~\ref{I-bd}, for every $\A \in \fM^*$ there exists $q \le Q$ such that
    \begin{align*}
	g(\alpha_3,\alpha_2) \ll \kappa(q) P(1+P^2|\beta_2| + P^3|\beta_3|)^{-1/2} + q^{2/3+\ep},
    \end{align*}
    and one easily confirms that the first term in this expression is the dominating one. It follows that
    \begin{align*}
	&\int_{\cM_3(Q,X)}|g(\alpha_3,\alpha_2)^{2+\delta} h(A\alpha_3+\lambda, \mu)^2| \dr \alpha_3 \\
	&\quad \ll \sum_{1 \le q \le Q}(\kappa(q)P)^{2+\delta} \sum_{\substack{a_3=1 \\ (\a,q)=1}}^q \int_{X}^{\infty} \frac{|h(B(a_3/q+\beta_3)+\lambda, \mu)|^2}{(1+P^2|\beta_2| + P^3|\beta_3|)^{1+\delta/2}} \dr \beta_3,
    \end{align*}
    and in a similar manner to the treatment in \cite{BW:07} we deduce that for every $\mu \in \R$ one has
    \begin{align*}
	\sum_{\substack{a_3=1 \\ (\a,q)=1}}^q |h(B(a_3/q+\beta_3)+\lambda,\mu)|^2  &\le \sum_{\substack{a_3=1 \\ (a_3,q)=1}}^q \sum_{x,y \in \cA} e((x^3-y^3)Ba_3/q)\\
	& \le |B| \sum_{1\le x,y\le P}(x^3-y^3,q) \ll P^2q^\ep q_3,
    \end{align*}
    where $q_3$ denotes the cubic kernel of $q$ defined via $q=q_0 q_3^3$ with $q_0$ cubefree. It follows that altogether we have
    \begin{align*}
	&\int_{\cM_3(Q,X)}|g(\alpha_3,\alpha_2)^{2+\delta} h(A\alpha_3+\lambda, \mu)^2| \dr \alpha_3 \\
      &\quad \ll P^{4+\delta} \sum_{q=1}^Q q^\ep (\kappa(q))^{2+\delta} q_3 \int_{X}^{\infty} (1+P^2|\beta_2| + P^3|\beta_3|)^{-1-\delta/2}\dr \beta_3\\
      & \quad \ll P^{1+\delta}X^{-\delta/2} \sum_{q=1}^{\infty} q^\ep (\kappa(q))^{2+\delta} q_3.
    \end{align*}
    Finally, the sum over $q$ converges whenever $\ep$ is small enough compared to $\delta$.

    In order to prove the second statement of the lemma, we observe that by H\"older's inequality we have
    \begin{align*}
	&\int_{\cM_3(Q,X)}|g(\alpha_3,\alpha_2) h(A\alpha_3+\lambda,  \mu)^6| \dr \alpha_3 \\
	&\ll \left(\int_{\cM_3(Q,X)} |g(\alpha_3,\alpha_2)^3 h(A\alpha_3+\lambda, \mu)^2| \dr \alpha_3\right)^{1/3} \left(\int_0^1 |h(A\alpha_3+\lambda,  \mu)|^8 \dr \alpha_3\right)^{2/3}.
    \end{align*}
    By considering the underlying equations it transpires that the second integral is bounded above by
    \begin{align*}
	\int_0^1 |h(\alpha_3,\alpha_2)|^8 \dr \alpha_3 \ll \int_0^1 |f(\alpha_3,0)|^8 \dr \alpha_3 \ll P^5,
    \end{align*}
    where we used Theorem~1 of \cite{Va:86a}. It now follows from (\ref{bw07lem}) that the expression in question is bounded above by $(P^2X^{-1/2})^{1/3} (P^5)^{2/3}  \ll P^4X^{-1/6}$ as claimed.
\end{proof}

\section{Proof of Theorem~\ref{quadcubthm}}

We now have the means at hand to complete the proof of Theorem~\ref{quadcubthm}. Our first task in this section is to obtain a sharper version of the Weyl-type estimate contained in Lemma~\ref{weylest}.

\begin{lemma}\label{weyl23}
    Suppose that $Q \le P^{3/4}$ and $\A \in \fm(Q)$. Then for all $M$-tuples $\j$ there exists an index $j_i$ with
    \begin{align*}
	|f_{j_i}(\A; [1,P])| \le P^{1+\ep}Q^{-1/3}.
    \end{align*}
\end{lemma}

\begin{proof}
    Fix $\j$, and suppose that for some $\A \in [0,1)^{r}$ one has $|f_{j_i}(\A)| \ge P^{1+\ep}Q^{-1/3}$ for each $1 \le i \le M$. Then by applying Theorem~5.1 of \cite{Bak:DI}, as in the argument of Lemma~5.2 of \cite{W:15sae5}, we find that there exist $q \le Q$ and $\tau > 0$ such that
    \begin{align*}
	\left\|q \gamma_{2,j_i}\right\|  \ll QP^{-2-\tau}  \quad \text{ and }\quad \left\|q \gamma_{3,j_i}\right\|  \ll QP^{-3-\tau} \qquad (1 \le i \le M).
    \end{align*}
    The invertibility of the coordinate transform implies, as in the proof of Lemma~\ref{weylest}, that for large enough $P$ one has
    \begin{align*}
	\left\|q\alpha_{2,i}\right\| \le QP^{-2} \quad (1 \le i \le r_Q) \quad \mbox{and} \quad
	\left\|q\alpha_{3,i}\right\| \le  QP^{-3} \quad (1 \le i \le r_C).	
    \end{align*}
    As before, we conclude that $\A$ must lie in $\fM(Q)$, and the enunciation follows.
\end{proof}

Recall the definitions \eqref{Idef} and \eqref{tNdef}. From now on set $Q=P^{3/4}$, and as before we let $X= P^{1/(6r)}$ for the asymptotic estimate and $X= (\log P)^{1/(6r)}$ for the lower bound. Recall the definition of $\fM$ and $\fN$ from Sections 3 and 4 and set $\cM(Q,X) =\fM(Q) \setminus \fN(X)$. In what follows, we will abbreviate $\tN_{s,\k,\M}(\fB) = \tN_s(\fB)$ and $I_{\u,\k,\M}(\cA) = I_{\u}(\cA)$ for simplicity.  Our first goal is to estimate $\tN_s(\fm(Q))$, where we have $\cA=[1,P]$.

Write $\fm^{(j)}$ for the set of $\A \in [0,1)^r$ for which $|f_j(\A)| \le P^{3/4+\ep}$, and let $\sigma = s-2s_0$. For any given $\sigma$-tuple $(j_1, \dots, j_{\sigma}) \in \{1, \dots, s\}$ the non-singularity condition implies that the remaining $2s_0$ variables may be  assembled into a mean value of the shape $I_{\u}([1,P])$, and thus Lemma~\ref{weyl23} implies that
\begin{align*}
    \tN_s(\fm^{(j_1)} \cap \dots \cap \fm^{(j_{\sigma})}) \ll P^{\frac34\sigma + \ep} I_{\u}([1,P]).
\end{align*}
Consider a fixed $\A \in \fm(P^{3/4})$. Lemma~\ref{weyl23} ensures that one can find an index $j_1 \in \{1,\dots,r\}$ with $\A \in \fm^{(j_1)}$. Iterating this procedure, after $k-1$ steps we can find an index $j_k \in \{1,\dots,r+k-1\} \setminus \{j_1,\dots,j_{k-1} \}$ with $\A \in \fm^{(j_k)}$. Since $\A \in \fm(P^{3/4})$ has been arbitrary, after $\sigma$ steps it follows that
\begin{align*}
    \fm(P^{3/4}) \subseteq \bigcup (\fm^{(j_1)} \cap \dots \cap \fm^{(j_{\sigma})}),
\end{align*}
where the union is over all $\sigma$-element subsets of $\{1,\dots,\sigma+r-1\}$. We may conclude that
\begin{align*}
    \tN_s(\fm(P^{3/4})) \ll P^{\frac34 \sigma + \ep} I_{\u}([1,P]).
\end{align*}

We first consider the case $r_Q=r_C=r/2$, so that $t=1$ and $\nu(1)=2$. Recalling Wooley's bound
\begin{equation}\label{WooleyJ5}
    J_{5,(2,3)}([1,P]) \ll P^{5+1/6+\ep}
\end{equation}
of \cite[Theorem~1.3]{W:15sae5}, Lemma~\ref{weyl23} together with Theorem~\ref{mvt} yields for $u_1=5$ that
\begin{align*}
    \tN_{s}(\fm(P^{3/4})) &\ll P^{\frac{3}{4}(s-2s_0) + \ep} (J_{5,(2,3)}([1,P]))^{r/2} \ll P^{\frac34(s-5r) + \ep} (P^{5+1/6+\ep})^{r/2}.
\end{align*}
Note that the exponent is smaller than $s-K = s-5r/2$ whenever $s>(16/3)r$, and since $(16/3)r = (32/3)r_Q = (32/3)r_C$ this is in line with the enunciation of the theorem.

In the cases with $r_Q \ne r_C$ we have
\begin{align}\label{par23}
    t=2, \quad k=3, \quad \nu(1) = \nu(2) = 1.
\end{align}
If $r_Q > r_C$ the parameters are given by
\begin{align}\label{par2>3}
    \mu_1=r_Q, \quad \mu_2=r_C, \qquad u_1=2, \quad u_2=5,
\end{align}
and we deduce from Theorem~\ref{mvt} that
\begin{align*}
    I_{(2,5)}([1,P]) \ll (J_{2,2}([1,P]))^{r_Q-r_C}(J_{5,(2,3)}([1,P]))^{r_C} \ll (P^{2+\ep})^{r_Q-r_C} (P^{5+1/6+\ep})^{r_C},
\end{align*}
where we used Hua's inequality and Wooley's bound (\ref{WooleyJ5}) as above. This shows
\begin{align*}
    \tN_{s}(\fm(P^{3/4})) \ll P^{\frac{3}{4}(s-(4r_Q + 6r_C)) + 3 r_C + 2 r_Q + r_C/6+\ep},
\end{align*}
and for $s> 4r_Q +(20/3)r_C$ the exponent is smaller than $s-(2r_Q+3r_C)$.

For $r_C>r_Q$  we have
\begin{align*}
    \mu_1=r_C, \quad \mu_2=r_Q \qquad u_1=4, \quad u_2=5,
\end{align*}
and in this case Wooley's bound (\ref{WooleyJ5}) together with Hua's Lemma~yields
\begin{align*}
    \tN_s(\fm(P^{3/4})) &\ll P^{\frac{3}{4}(s-(8r_C + 2r_Q)) + \ep} (J_{4,3}([1,P]))^{r_C-r_Q}(J_{5,(2,3)}([1,P]))^{r_Q} \\
    &\ll P^{\frac{3}{4}(s-(8r_C + 2r_Q)) + \ep} (P^{5+\ep})^{r_C-r_Q} (P^{5+1/6+\ep})^{r_Q},
\end{align*}
which is acceptable whenever $s>8r_C+(8/3)r_Q$.\\

In the case $\cA = \cA(P,R)$ the analysis is more delicate, due to the fact that we have only a limited number of complete exponential sums at our disposal. In this case we take
\begin{align}\label{par3>2}
    \mu_1=r_C, \quad \mu_2=r_Q, \qquad u_1=3, \quad u_2=5,
\end{align}
and we aim to prove the theorem with $s = 7r_C+\lceil(11/3)r_Q \rceil$.  We write $\Delta=r_C-r_Q$ and let
\begin{align*}
     N_s^*(\fB) =\int_{\fB} \prod_{j=1}^{6\Delta}  h_j(\A) \prod_{j=6\Delta +1}^{s}g_j(\A) \, \dr \A,
\end{align*}
where $g_j(\A)=g(\G_j)$ and $h_j(\A)=h(\G_j)$, and $g(\A)$ and $h(\A)$ are as in the preamble to Lemma~\ref{prunesm} with $\cA=\cA(P,R)$. By considering the underlying diophantine equations, one finds that the number of solutions of the system (\ref{sys23}) with $\x \in [1,P]^s$ is bounded below by $N_s^*([0,1)^r)$, whence it suffices to establish a lower bound for the latter quantity.
It follows from \cite[Theorem~1.2]{W:00cubes} that for a suitable choice of $R$ there exists a number $\tau>0$ satisfying
\begin{equation}\label{W3cubes}
    J_{3,3}(P; \cA(P,R)) \ll P^{3+1/4-\tau},
\end{equation}
and we note for further reference that the current bounds imply $\tau < 1/24$.  Let $\fm^{(j)}$ denote the set of $\A \in [0,1)^r$ for which $|g_j(\A)| \le P^{3/4+\ep}$. By Lemma~\ref{weyl23}, one has
\begin{align*}
    \fm \subseteq \fm^{(6\Delta+1)} \cup \cdots \cup \fm^{(7\Delta+r_Q)},
\end{align*}
so after re-indexing and summing over $j$, we find that $N_s^*(\fm)$ is bounded above by a sum of at most $r_C$ expressions of the shape
\begin{align*}
     P^{\tau(4\Delta - 1) + \ep} \int_{[0,1)^r} \prod_{j=1}^{6\Delta}  |h_j(\A) | \prod_{j=6\Delta+1}^{7\Delta} |g_j(\A)|^{1-4\tau} \prod_{j=7\Delta+1}^{s} |g_j(\A)| \, \dr \A.
\end{align*}
We now apply \eqref{trivineq} in such a way that, for some sets of indices $\cJ_1$, $\cJ_2$, and $\cJ_3$ with $|\cJ_1|=|\cJ_2|=\Delta$ and $|\cJ_3|=r_Q$,  one has
\begin{align*}
    N_s^*(\fm) \ll &P^{\psi+\ep}\int_{[0,1)^r} \bigg(\prod_{j\in \cJ_1} |h_j(\A)|^6\bigg) \bigg( \prod_{j \in \cJ_2} |g_{j}(\A)|^{1-4\tau}\bigg) \bigg( \prod_{j\in \cJ_3} |g_j(\A)|^{32/3} \bigg) \dr \A.
\end{align*}
Here we have written
\begin{align}\label{psidef}
    \psi=\tau(4\Delta-1)+\lceil \tfrac{2}{3} r_Q \rceil - \tfrac{2}{3} r_Q,
\end{align}
and we have used the fact that $s=7\Delta+10r_Q+\lceil\frac{2}{3}r_Q\rceil$.
We next apply \eqref{gamalp2}. Writing $\L_i = (\lambda_{2,i},\lambda_{3,i})$, where $\lambda_{k,i}$ is a linear combination of the $\gamma_{k,l}$ with $l \ne i$, we find that
\begin{align}
\label{smoothminor}
  N_s^*(\fm) &\ll P^{\psi+\ep} \int_{[0,1)^r} \bigg(\prod_{j\in \cJ_2} \left|h(\G_j+\L_j)^6 g(\G_j)^{1-4\tau}\right|\bigg) \bigg( \prod_{j \in \cJ_3}|g(\G_j)|^{32/3} \bigg) \dr \G \notag \\
    &\ll \left(   \sup_{\L \in \R^2} \sup_{\gamma_2 \in [0,1)} \int_{[0,1)}\left|h(\G + \L)^6 g(\G)^{1-4\tau}\right| \dr \gamma_3 \right)^{\! \Delta}  \left( \int_{[0,1)^2} |g(\G)|^{32/3} \dr \G \right)^{\! r_Q}.
\end{align}
It follows from \cite[Theorem~1.3]{W:15sae5} that the second integral is bounded above by $P^{17/3+\ep}$. Meanwhile, upon abbreviating $\fM_i(Q)$ by $\fM_i$, we also have
\begin{align*}
    \sup_{\gamma_2 \in [0,1)} \int_{[0,1)}\left|h(\G + \L)^6 g(\G)^{1-4\tau}\right| \dr \gamma_3 & \ll \sup_{\G \in \fm^*} |g(\G)|^{1-4\tau} \int_{[0,1)} \left| h(\G + \L) \right|^6 \dr \gamma_3 \\
    & \quad + \sup_{\gamma_2 \in \fM_2}\int_{\fM_3}\left|h(\G + \L)^6 g(\G)^{1-4\tau}\right| \dr \gamma_3.
\end{align*}
In the first term, (\ref{W3cubes}) together with Lemma~5.2 of \cite{W:15sae5} yields
\begin{align*}
    \sup_{\G \in \fm^*} |g(\G)|^{1-4\tau} \int_{[0,1)} \left| h(\G + \L) \right|^6 \dr \gamma_3 \ll P^{\frac34 - 3 \tau + \ep} \int_0^1 |h(\G)|^6 \dr \G \ll P^{4-4\tau+\ep}.
\end{align*}
In order to estimate the contribution from the major arcs we observe that an application of H\"older's inequality yields
\begin{align*}
   &\int_{\fM_3}\left|h(\G + \L)^6 g(\G)^{1-4\tau}\right| \dr \gamma_3  \ll  \left(\int_{\fM_3}|g(\G)^{5/2} h(\G+\L)^2| \dr \gamma_3\right)^{\omega_1} \left(\int_0^1 |h(\G)|^{ \phi}\dr \gamma_3\right)^{\omega_2},
\end{align*}
where $\omega_1 = (2-8\tau)/5$, $\omega_2 = (3+8\tau)/5$, and $\phi =  (26+16\tau)/(3+8\tau)$. Observe in particular that for $\tau < 1/24$ one has $\phi >8$. It follows that the first integral is $O(P^{3/2})$ by Lemma~\ref{prunesm} \eqref{bw07lem}, and  the second one is $O(P^{\phi-3+\ep})$ by Hua's Lemma, whence we obtain an overall contribution of
\begin{align*}
    P^{(3/5)(1-4\tau)} P^{(\phi-3)(3+8\tau)/5+\ep} \ll  P^{4-4\tau+\ep}
\end{align*}
from the major arcs. Together with the minor arcs contribution we find
\begin{align*}
    \sup_{\gamma_2 \in [0,1)} \int_{[0,1)}\left|h(\G + \L)^6 g(\G)^{1-4\tau}\right| \dr \gamma_3 \ll P^{4-4\tau+\ep},
\end{align*}
and therefore
\begin{align*}
    \int_{[0,1)^r} \bigg(\prod_{j\in \cJ_2}|h(\G_j+\L_j)|^6 |g(\G_j)|^{1-4\tau}\bigg) \bigg(\prod_{j\in\cJ_3} |g(\G_j)|^{32/3}\bigg) \dr \A  \ll P^{4\Delta(1-\tau)+\ep}P^{(17/3)r_Q}.
\end{align*}
On recalling (\ref{psidef}) and (\ref{smoothminor}), we thus obtain
\begin{align*}
    N_s^*(\fm) \ll P^{\tau(4\Delta-1)+\ep}  P^{\lceil \frac23 r_Q \rceil - \frac 23 r_Q}  P^{4r_C+(5/3)r_Q-4\tau\Delta} \ll P^{4r_C + \lceil \frac{5}{3}r_Q \rceil - \tau/2}
\end{align*}
for $\ep$ sufficiently small.

It follows from our definitions of major and minor arcs that
\begin{align}\label{pruning}
    N_s(\fn(X)) \ll N_s(\fm(Q)) + N_s(\cM(Q,X)),
\end{align}
where $N_s(\fB)$ denotes either $\tN_s(\fB)$ or $N_s^*(\fB)$.
In view of the preceding estimates, (\ref{pruning}) shows that the analysis of the minor arcs $\fn(X)$ will be complete upon obtaining a satisfactory bound for $N_{s}(\cM(Q,X))$.

\begin{lemma}\label{minor23}
    Let $1 \le X \le Q^{1/(2M)} $ be arbitrary, and suppose that the system \eqref{sys23} is highly non-singular with $s$ given via \eqref{s0def} where $\u$ is as in \eqref{par2>3} or \eqref{par3>2}. Then we have
    \begin{align*}
	 N_{s}(\cM(Q,X)) \ll P^{s-K}X^{-1/(6M)}.
    \end{align*}
\end{lemma}
\begin{proof}
    The relation \eqref{gamalp2} implies that, whenever $\A \in \fM(Q)$, then for every pair of indices $i,j$ there exists an integer $b_{i,j}$ with $|b_{i,j}| \le B = r \max_{i,j}\{c_{ij}, d_{ij}\}$ such that $\gamma_{k_i,j} - b_{i,j} \in \fM_{k_i}(Q)$. We show that when $\A \not\in \fN(X)$, then necessarily one has $\G \not\in \fN(Y) + \Z^r$ for $Y=X^{1/M}$. For this purpose, observe that for $i\in \{1,2\}$ only $\mu_i$ of the entries of $\G_{k_i}$ are independent. Now suppose that all entries of $\G_{k_i}$ are in $\fN_{k_i}(Y)+ \Z^r$.  Then for every $j \in \{1, \dots, \mu_i\}$ there exist $q_j \le X$ and $\Lambda_{k_i,j} \in \Z$ with $|\gamma_{k_i,j} - \Lambda_{k_i,j}/q_j| \le YP^{-k_i}$.  The invertibility of the coordinate transform \eqref{gamalp2} implies that we may retrieve the $\A_{k_i}$ from the $\G_{k_i}$, and we therefore deduce that
    \begin{align*}
	| \alpha_{k_i,l}  - a_{k_i,l}/q| \le \kappa YP^{-k_i} \qquad (1 \le l \le \mu_i)
    \end{align*}
    for some constant $\kappa$ depending at most on the $(c_{ij})$ and $(d_{ij})$.
    However, we have $q = q_1\cdots q_{\mu_i}\le Y^{\mu_i}$, whence
    \begin{align*}
	\alpha_{k_i,j} \in \fN_{k_i}(Y^{\mu_i}) \subseteq \fN_{k_i}(Y^{M}) = \fN_{k_i}(X) \qquad (1 \le j \le r_i).
    \end{align*}
    It follows that whenever $\A \in \cM(Q,X)$, then there exists some pair of indices $(k_i,j)$ with $\gamma_{k_i,j} \in \fn_{k_i}(Y)+ \Z$, so altogether $\G \in \cM(Q,Y)+  \Z^r \cap [-B,B]^r$.

    For the rest of the argument we abbreviate $\fM = \fM(Q)$, $\fN = \fN(Y)$ and $\cM = \cM(Q,Y)$, and we use the same conventions for the respective symbols when equipped with suffices or asterisks.

    For $r_Q \ne r_C$ and $\cA=[1,P]$ set $v_i = 2u_i + (s-2s_0)/M$ for $i\in \{1,2\}$, so that $(\mu_1-\mu_2)v_1 + \mu_2v_2=s$. The relations \eqref{trivineq} and \eqref{gamalp2} together with the above argument imply that there exist sets of indices $\cJ_1$ and $\cJ_2$, with $|\cJ_1|=\mu_1-\mu_2$ and $|\cJ_2|=\mu_2$, such that
    \begin{align*}
	   \tN_{s}(\cM(Q,X)) &\ll \int_{\cM(Q,X)} \prod_{i\in \cJ_1} |f(\G_i)|^{v_1} \prod_{j\in \cJ_2} |f(\G_j)|^{v_2}\dr \A\\
	   &\ll \int_{\cM(Q,Y)} \prod_{i \in \cJ_1} |f(\G_i)|^{v_1} \prod_{j \in \cJ_2}|f(\G_j)|^{v_2}\dr \G.
    \end{align*}
    Note that by the non-singularity condition we may assume that all entries $(\gamma_{k_2, i})_{i \in \cJ_1}$ are determined by the entries $(\gamma_{k_2, j})_{j \in \cJ_2}$. We therefore obtain
    \begin{align*}
	\tN_s(\cM(Q,X)) &\ll \int_{(\fM^*)^{\mu_2-1} \times \cM^*} T_{\cJ_1}(\fM_{k_1}^{\mu_1-\mu_2},\G_{k_2}) \prod_{j \in \cJ_2}|f(\G_j)|^{v_2} \dr\G_j\\
	  & + \int_{(\fM^*)^{\mu_2}} T_{\cJ_1}(\fM_{k_1}^{\mu_1-\mu_2-1} \times \cM_{k_1},\G_{k_2})\prod_{j \in \cJ_2}|f(\G_j)|^{v_2}  \dr\G_j ,
    \end{align*}
    where we wrote
    \begin{align*}
         T_{\cJ_1}(\fB,\G_{k_2})= \int_{\fB} \prod_{i \in \cJ_1} |f(\gamma_{k_1, i},\gamma_{k_2, i})|^{v_1} \dr \gamma_{k_1, i}.
    \end{align*}
    Observe that we have $v_2>2u_2=10$ regardless of which of $r_C$ or $r_Q$ is larger. Writing
    \begin{align*}
    	T(\fC) = \sup_{\gamma_{k_2} \in \fM_{k_2}} \int_{\fC} |f(\gamma_{k_1},\gamma_{k_2})|^{v_1} \dr \gamma_{k_1},
    \end{align*}
    we may therefore deploy Lemma~\ref{aux}  with $v_1 > 2u_1 = 4$ to obtain
    \begin{align*}
        \tN_{s}(\cM(Q,X)) &\ll T(\fM_{k_1})^{\mu_1-\mu_2} \bigg( \int_{\fM^*} |f(\G)|^{v_2} \dr \G\bigg)^{\mu_2-1} \int_{\cM^*} |f(\G)|^{v_2} \dr \G \;  \\
        & \phantom{AA} +  T(\fM_{k_1})^{\mu_1-\mu_2-1}T(\cM_{k_1}) \bigg( \int_{\fM^*} |f(\G)|^{v_2} \dr \G\bigg)^{\mu_2} \\
        &\ll P^{(v_2-5)\mu_2 +(v_1-k_1)(\mu_1 - \mu_2) }Y^{-1}.
    \end{align*}
    Upon noting that $s=v_1(\mu_1-\mu_2)+v_2\mu_2$ and $K=5\mu_2+k_1(\mu_1-\mu_2)$, this yields the desired conclusion.

    Similarly, for $r_Q=r_C=r/2$ we deduce from \eqref{trivineq} and Lemma~\ref{aux} (ii) that
    \begin{align*}
	\tN_{s}(\cM(Q,X)) &\ll  P^{s-2s_0}\int_{(\fM^*)^{r/2-1} \times \cM^*} \bigg(\prod_{i=1}^{r/2} |f(\G_i)|^{10}\bigg) \dr \G  \ll P^{s-2s_0+(5/2)r}Y^{-1},
    \end{align*}
	 and the result follows on noting that $s_0 = K = 2r_Q + 3r_C=(5/2)r$ in this case.
	
    Finally, in the smooth case we have $|\cJ_1|=\Delta=r_C-r_Q$ and $|\cJ_2|=r_Q$, and as in the argument leading to (\ref{smoothminor}) we obtain
    \begin{align*}
	N^*_{s}(\cM(Q,X)) &\ll P^{\lceil(2/3)r_Q\rceil}  \int_{\cM(Q,Y)} \prod_{i \in \cJ_1} |h(\G_i+\L_i)^{6} g(\G_i)| \prod_{j\in \cJ_2}|g(\G_j)|^{10}\dr \G
    \end{align*}
    for suitable vectors $\L_i = (\lambda_{2,i}, \lambda_{3,i}) \in \R^2$, where $\lambda_{k,i}$ is a linear combination of the coefficients $\gamma_{k,l}$ with $l \ne i$. By writing
    \begin{align*}
	\widehat T_{\cJ_1} (\fB,\G_{2})= \sup_{\L \in \R^{2\Delta}}\int_{\fB} \prod_{i \in \cJ_1} |h(\G_i + \L_i)^{6} g(\G_i)| \dr \gamma_{3, i},
    \end{align*}
    we see that
    \begin{align*}
	N^*_{s}(\cM(Q,X)) &\ll P^{\lceil(2/3)r_Q\rceil} \int_{(\fM^*)^{r_Q-1} \times \cM^*} \widehat T_{\cJ_1}(\fM_3^{\Delta},\G_{2}) \prod_{j\in \cJ_2}|g(\G_{j})|^{10} \dr\G_{j}\\
	& \quad + P^{\lceil(2/3)r_Q\rceil}\int_{(\fM^*)^{r_Q}} \widehat T_{\cJ_1}(\fM_3^{\Delta-1} \times \cM_3,\G_{2})\prod_{j\in \cJ_2}|g(\G_{j})|^{10}  \dr\G_{j}.
    \end{align*}
    Letting $$\widehat T(\fC) = \sup_{\gamma_2 \in \fM} \sup_{\L \in \R^2} \int_{\fC} |h(\G+\L)^6g(\G)| \dr \gamma_3,$$ then by an argument analogous to the one above with Lemma~\ref{prunesm} \eqref{bw07lem2} in the place of Lemma~\ref{aux}(i) we obtain
    \begin{align*}
        N_{s}^*(\cM(Q,X)) &\ll P^{\lceil(2/3)r_Q\rceil} \widehat T(\fM_3)^{\Delta} \bigg( \int_{\fM^*} |f(\G)|^{10} \dr \G\bigg)^{r_Q-1} \int_{\cM^*} |f(\G)|^{10} \dr \G \;  \\
         & \phantom{AAAAA} +  P^{\lceil(2/3)r_Q\rceil} \widehat T(\fM_3)^{\Delta-1}\widehat T(\cM_3) \bigg( \int_{\fM^*} |f(\G)|^{10} \dr \G\bigg)^{r_Q} \\
        &\ll P^{\lceil(2/3)r_Q\rceil}  P^{4\Delta}   P^{5r_Q}Y^{-1/6} \ll P^{s-K}Y^{-1/6}.
    \end{align*}
    This completes the proof of the lemma.
\end{proof}

The analysis of the minor arcs $\fn(X)=[0,1]^r \setminus \fN(X)$ is now completed on inserting Lemma~\ref{minor23}, together with the estimates ensuing from Lemma~\ref{weyl23}, into \eqref{pruning}.

For the major arc analysis, only small modifications to the arguments of Section 4 are required.  In completing the singular series, we again have to make a case distinction as to whether $r_Q>r_C$ or not. If $r_Q>r_C$ we have \eqref{par23} and \eqref{par2>3}, and with these parameters \eqref{A-bd} becomes
\begin{align}\label{A-bd2>3a}
 A(p^i)&  \ll p^{-i\left(\frac13(10+1/r) - 2\right) + \ep} + p^{-i\left(\frac13(4+1/r) - 1\right) + \ep} \ll p^{-\frac{i}{3}(1+1/r)+\ep},
\end{align}
and it follows that
\begin{equation}\label{sumAbd}
    \sum_{i=3}^{\infty} A(p^i) \ll p^{-1-1/r+\ep}.
\end{equation}
For $i\in \{1,2\}$ we make recourse to the estimate \eqref{Sqaest2}.
Following through the argument of the proof of Lemma~\ref{ss}, one thus obtains
\begin{align}\label{A-bd2>3b}
    A(p^i) &\ll p^{-i\left(\frac12(4+1/r) - 1\right) + \ep} +p^{-i\left(\frac12(10+1/r) - 2\right) + \ep}  \ll p^{-1-1/(2r)+\ep} \quad (i=1,2).
\end{align}
Now on combining \eqref{sumAbd} and \eqref{A-bd2>3b} one has for a suitable constant $c$ that
\begin{align*}
    \fS \le \prod_p(1+cp^{-1-1/(3r)}) \ll 1.
\end{align*}

If $r_C>r_Q$ we have \eqref{par3>2}, and thus
\begin{align*}
    A(p^i)&  \ll p^{-i\left(\frac13(10+1/r) - 2\right) + \ep} + p^{-i\left(\frac13(6+1/r) - 1\right) + \ep} \ll p^{-i(1+1/(3r))+\ep},
\end{align*}
which is also satisfactory.  Finally, in the case $r_Q=r_C$ the bound is given by the first term in \eqref{A-bd2>3a}, whence $A(p^i) \ll p^{-4i/3}$ for all $i$.
It follows that the singular series converges also in the setting of Theorem~\ref{quadcubthm}, and also that $\fS-\fS(X) \ll X^{-\delta}$ for some $\delta > 0$.

For the singular integral, the results of Lemma~\ref{si} are satisfactory even in the case of Theorem~\ref{quadcubthm}.  To verify this, we first observe that when $r_Q \neq r_C$ one has $u_1 \ge 2> 3/2=(k/2)\nu(1)$ and $u_2=5>3=(k/2)(\nu(1)+\nu(2))$, whereas when $r_Q=r_C$ we have $u_1=5 > 3 = (k/2)\nu(1)$.  The proof of Theorem~\ref{quadcubthm} is now complete on recalling the concluding discussion of Section 4.

\bibliographystyle{amsplain}
\bibliography{fullrefs}

\end{document}